\documentclass[11pt]{article}
\usepackage{amsmath,amsthm,amsfonts,amssymb}

\usepackage{mathtools}
\usepackage{bbm}
\usepackage{graphicx} 
\usepackage{comment}
\usepackage{fourier}
\usepackage{soul}
\usepackage{color}
\usepackage{enumerate}
\usepackage[margin=1in]{geometry}
\usepackage{float}

\newtheorem{theorem}{Theorem}[section]
\newtheorem{proposition}{Proposition}

\newtheorem{lemma}{Lemma}

\theoremstyle{definition}
\newtheorem{definition}{Definition}

\theoremstyle{remark}
\newtheorem{remark}{Remark}

\newcommand{\hmu}{\widehat{\mu}}
\newcommand{\W}{\mathcal{W}}
\newcommand{\1}{\mathbbm{1}}
\newcommand{\R}{\mathbb{R}}

\newcommand{\Z}{\mathbb{Z}}

\AtEndDocument{%
  \bigskip
  {\footnotesize
  \textsc{Department of Mathematics, The University of Haifa at Oranim, Tivon 36006, Israel}\par
  \textit{E-mail address} \texttt{amir.algom@math.haifa.ac.il}\par
  }%

  \bigskip
  {\footnotesize
  \textsc{Einstein Institute of Mathematics, Hebrew University, 91904 Jerusalem, Israel.}\par
  \textit{E-mail address} \texttt{snir.benovadia@mail.huji.ac.il}\par
  }%

  \bigskip
  {\footnotesize
  \textsc{Department of Mathematics, The Pennsylvania State University, University Park, PA 16802, USA}\par
  \textit{E-mail address} \texttt{hertz@math.psu.edu}\par
  }%

  \bigskip
  {\footnotesize
  \textsc{Facultad de Ingenier\'ia, Universidad de la Rep\'ublica, Julio Herrera y Reissig 565, Montevideo 11300, Uruguay}\par
  \textit{E-mail address} \texttt{shannon@cmat.edu.uy}\par
  }%
}%

\title{Fourier decay and non-decay for pseudo-affine self-conformal measures}
\date{}
\author{Amir Algom, Snir Ben Ovadia, Federico Rodriguez Hertz, and Mario Shannon}

\begin{document}
\maketitle
\begin{abstract}
We study the sharpness of recent sufficient conditions for polynomial Fourier decay of self conformal measures on the line. First, we construct a $C^\infty$ iterated function system which is not $C^1$-conjugate to  self-similar, but which nevertheless admits a stationary measure that is not Rajchman. Second, for every strongly separated Bernoulli convolution $\mu$ and every $1\leq r<\infty$, we construct a $C^r$-diffeomorphism $h$ such that $h'$ is constant on $\operatorname{supp}\mu$, yet the image measure $h\mu$ has polynomial Fourier decay. All constructions are within the framework of  pseudo-affine iterated function systems, previously introduced by the authors.
\end{abstract}

\section{Introduction}
\subsection{Main results}
Given a Borel probability measure $\mu \in \mathcal{P}(\mathbb{R})$, its Fourier transform is defined as
\begin{equation*}
    \hmu(\xi)=\int_\R e^{-2\pi i\xi t} d\mu(t),\, \xi \in \mathbb{R}.
\end{equation*}
The measure $\mu$ is said to be \emph{Rajchman} if 
$$\lim_{|\xi|\rightarrow \infty} \hmu(\xi) =0.$$ 
In this paper we study the Fourier decay problem:  determine, in terms of the geometry of the measure $\mu$, whether it is Rajchman or not; and, if Rajchman, at which rate $\hmu(\xi)$ decays to zero. 
Of specific interest is whether there exists some $\alpha>0$ such that $\lVert \hmu(\xi) \rVert \lesssim |\xi|^{-\alpha}$, where $\lesssim$ means this inequality holds up to a global constant independent of $\xi$; in this case we say that $\mu$ has polynomial Fourier decay. This last condition is  equivalent to  $\mu$ having positive Fourier dimension \cite{Ekstrom2017jorg}. Deciding whether or not such estimates hold true is a classical problem that, beyond intrinsic interest,  is known to have many applications across various fields. See Lyon's survey \cite{Lyons1995survey} for an overview of the Rajchman property, and Mattila's book \cite{Mattila2015Fourier} for a general introduction to the role of Fourier analysis in geometric measure theory.

We will work  with self-conformal measures on $[0,1]$, which serve as simple models for more general stationary measures.  A measure $\mu \in \mathcal{P}([0,1])$ is called \emph{self-conformal} if there exists a strictly positive probability vector $\mathbf{p}=(p_1 ,\dots,p_n)$ such that
\begin{equation} \label{eq: self-similar measure}
    \mu=\sum_{i=1} ^n p_i\cdot f_i\mu,\,\text{ where all }f_i\in C^{1}([0,1]) \text{ and } |f_i'|\in (0,1),
\end{equation}
where $f_i \mu$ is the push-forward of $\mu$ via $f_i$. The family of maps $\Phi:=\{ f_1,\dots,f_n\}$ is called a \emph{self-conformal} IFS (Iterated Function System). If all the maps in $\Phi$ are affine, we call it a \emph{self-similar IFS} and the measure $\mu$ a \emph{self-similar measure}. The assumptions in \eqref{eq: self-similar measure} are known to imply that there exists a unique compact set $\emptyset \neq X=X_\Phi \subseteq [0,1]$ such that
\begin{equation} \label{eq: union}
    X=\bigcup_{i=1}^n f_i(X).
\end{equation}
We call $X$ the \emph{attractor} of $\Phi$ or the self-conformal set corresponding to it. We always assume $X$ is infinite, and to this end it suffices  that there exist $f_i,f_j\in \Phi$ that don't share the same fixed point; this will be assumed throughout the paper. Notice that this is immediate if the union in \eqref{eq: union} is disjoint, in which case the attractor is a Cantor set and we say the IFS $\Phi$ has the SSC (Strong Separation Condition).

The family of \emph{Bernoulli convolutions}, which are defined as follows, provides some illuminating examples:
given $\lambda\in(0,1)$,  $\mu_\lambda$ is the self-similar
measure associated to the IFS
$\{x\mapsto \lambda x,\; x\mapsto \lambda x+1-\lambda\}$
with equal weights $(1/2,1/2)$.  By classical results of Erd\H{o}s and
Salem, $\mu_\lambda$ is not Rajchman if and only if
$\lambda^{-1}$ is a Pisot number (unless $\lambda = \frac{1}{2}$, in which case $\mu_\lambda=\text{Leb}|_{[0,1]})$; see
Section~\ref{section: background} for further discussion, and Section \ref{Section: sketch} for a proof sketch. More
generally, failure of the Rajchman property for a self-similar measure
on the line imposes strong arithmetic restrictions on the underlying
IFS; see
\cite{bremont2019rajchman,li2019trigonometric}.

Recent research has identified two main mechanisms that lead to the
Rajchman property for self-conformal measures on the line, and often
to positive Fourier dimension. The first is based on
 non-linearity of the underlying IFS. An IFS $\Psi$ of class
$C^2([0,1])$ is called \emph{linear} if
\begin{equation} \label{eq: linear}
    h''(x)=0 \text{ for every } h\in \Psi \text{ and every } x\in X_\Psi.
\end{equation}
Every self-similar IFS is linear in the sense of
\eqref{eq: linear}. However, as we showed in
\cite{ABRHS_pseudo-Affine_IFS}, there exist linear IFSs that are not
 $C^1$-conjugate to a self-similar IFS; see
Section~\ref{Section: sketch} for further discussion. Now, 
in the
real-analytic setting, it is known that every self-conformal measure with respect to a
$C^\omega([0,1])$ IFS $\Phi$ has positive Fourier dimension, if $\Phi$ is not $C^\omega$-conjugate to a self-similar IFS. In the
$C^2$ setting, the same conclusion holds, provided that $\Phi$ is not
$C^2$-conjugate to a linear IFS in the sense of
\eqref{eq: linear}. Recall that $\Phi$ is $C^r$-conjugate to an IFS
$\Psi$ if there exists a $C^r$ diffeomorphism $h$ of the interval such
that
$\Psi=\lbrace h\circ f_i\circ h^{-1}\rbrace_{f_i\in\Phi}.$ We will discuss the background of these results in detail in Section~\ref{section: background}.

 It is thus natural to ask: can \emph{linear IFS} be replaced by
\emph{self-similar IFS} in the Fourier-decay criterion recalled
above? that is, if a $C^r$ IFS is not $C^r$-conjugate to 
self-similar, must all  self-conformal measures be Rajchman?
This is the case if $r=\omega$. Our
first main result shows that it is not true in general:

\begin{theorem}
    \label{thm_main_non-conjugated}
There exists a self-conformal IFS $\Phi$ on $[0,1]$ such that:
\begin{enumerate}
    \item $\Phi$ is $C^\infty ([0,1])$;
    \item $\Phi$ is not $C^1$-conjugate to a self-similar IFS;
    \item There exists a non-trivial self-conformal measure with respect
    to $\Phi$ that is not Rajchman.
\end{enumerate}
\end{theorem}

The IFS constructed in Theorem~\ref{thm_main_non-conjugated} is linear
in the sense of \eqref{eq: linear}. To the best of our knowledge, this is the first
example of a non-Rajchman stationary measure for a self-conformal IFS
on the line that is not $C^1$-conjugate to a self-similar IFS.

The second  mechanism alluded to above is the \(L^2\)-flattening method. Roughly speaking, $L^2$-flattening is the property that the Fourier transform is small outside a sparse set of frequencies. When combined with suitable non-triviality of the derivatives, this can often be upgraded to polynomial Fourier decay. A headline application of the $L^2$-flattening method is the following result:
\begin{equation} \label{eq: Kaufman}
\text{If }  \mu \in \mathcal{P}([0,1]) \text{ is self-similar, and } g \in C^2(\mathbb{R}) \text{ has }g''>0, \text{ then } g \mu \text{ has polynomial Fourier decay}  
\end{equation}
This is remarkable since $\mu$ itself can fail to be even Rajchman, as can be seen if we choose, for instance, the Bernoulli convolution $\mu_\lambda$ with $\lambda^{-1}>1$ Pisot. 
Note that in \eqref{eq: Kaufman} it is natural to impose some non-linearity assumptions on $g$, since  if for instance $g$ is affine, then clearly $g\mu$ is not Rajchman if $\mu$ is not Rajchman. Most  proofs of \eqref{eq: Kaufman} use the non-trivial fact that $\mu$ is $L^2$-flattening; and, to make use of this property, it is crucial that $g'$ can at least be made locally invertible on $\text{supp}(\mu)$ with a H\"older inverse. We  refer to Section \ref{section: background} for more background about the $L^2$-flattening property, its abundance,  its applications, and about \eqref{eq: Kaufman} specifically.

Our second main result  exhibits examples where the phenomenon in \eqref{eq: Kaufman} still holds true, but nonetheless $g'$ is fully constant on the attractor. In particular,  standard applications of $L^2$-flattening are not available. 

 \begin{theorem}
\label{thm_main-conjugated_poly-decay}
Fix $0<\lambda<1/2$. For every $r\in \mathbb{N}, 0<\alpha\leq 1$ there exists a $C^{s}([0,1])$ diffeomorphism  $g$, where $s=r+\alpha$, such that:
\begin{enumerate}
    \item $g \mu_\lambda$ is a self-conformal measure;
    \item There exists some $\eta>0$ such that $g'\equiv \eta$ on $\text{supp}(\mu_\lambda)$.
    
    \item There exists some $\tau=\tau(s, \lambda)$ such that $\lVert \widehat{g \mu_\lambda} (\xi) \rVert\lesssim |\xi|^{-\tau}$; one may take 
    $$\tau(s,\lambda) = \frac{2\pi\,\lambda^{2s}} {(18e+\pi)s\log(1/\lambda)} >0.$$
\end{enumerate}
This is true even if $\lambda^{-1}$ is a Pisot number and therefore $\mu_\lambda$ itself is not  Rajchman.
 \end{theorem}
 We remark that  the map $g$ is tailored to $\mu_\lambda$; this is unlike \eqref{eq: Kaufman} and its generalizations, which give a uniform decay exponent for all strictly convex $C^2$ diffeomorphisms. In fact, the proof can be extended to \(C^\infty\) phases, with stretched-exponential decay in \(\log |\xi|\) (but not polynomial decay); see Proposition \ref{prop_smooth_stretched_log_decay}. When $\lambda^{-1}$ is a Pisot number, this loss of polynomial decay is unavoidable: Proposition~\ref{prop_flatness_obstruction_Pisot} shows that no $C^\infty$ diffeomorphism whose derivative is constant on $\operatorname{supp}(\mu_\lambda)$ can produce polynomial Fourier decay.

\begin{remark} \label{rmk: ekstorm}
Upon completing this work, we became aware of a closely related result of
Ekstr\"om \cite[Theorem~4]{Ekstrom2016random}. Applied to the special case of
$\mu_\lambda$, his theorem yields for every $s\in [1,\infty)$ a random $C^s$ diffeomorphism $g_\omega$
such that, almost surely,
$\dim_F(g_\omega\mu_\lambda)
\geq
\frac{\dim_H\mu_\lambda}{s}
=
\frac{\log 2}{s\log(1/\lambda)}
>0.$
Moreover, it follows directly from his construction that
$g_\omega'\equiv 1
\text{ on }\operatorname{supp}(\mu_\lambda).$
Thus, the existence of a $C^r,r\in [1,\infty),$ diffeomorphism whose
derivative is constant on $\operatorname{supp}(\mu_\lambda)$ and whose
pushforward has polynomial Fourier decay was already established by
Ekstr\"om; we do not claim novelty for this statement in isolation.

The additional contributions of our approach are threefold. First,
Ekstr\"om does not formulate or analyze the induced self-conformal
structure. Our method gives the conjugated system a particularly
explicit pseudo-affine structure (see Definition \ref{defn_pseudo-affine} below), and yields a precise Riesz-product
representation of the Fourier transform of $g\mu_\lambda$. Second, although both constructions are random, ours uses the randomness more "economically": a single random variable is chosen at each level and applied uniformly across all cylinders of that level. By contrast, Ekstr\"om assigns independent random increments to the individual complementary intervals. This additional freedom allows his construction to work in  greater generality, whereas our level-homogeneous organization makes the resulting diffeomorphism easier to describe and analyze, and may be more amenable to a future deterministic realization.
 Third, our method has a $C^\infty$ variant,
yielding stretched-exponential decay in $\log|\xi|$; see
Proposition~\ref{prop_smooth_stretched_log_decay}. When $\lambda^{-1}$
is a Pisot number, Proposition~\ref{prop_flatness_obstruction_Pisot}
shows that polynomial Fourier decay is impossible at $C^\infty$
regularity within the class of diffeomorphisms whose derivative is
constant on $\operatorname{supp}(\mu_\lambda)$. Ekstr\"om's construction
does not provide such a $C^\infty$  analogue.
We compare the two constructions further in
Section~\ref{Section: sketch}.
\end{remark}

Finally, to complement our discussion of \eqref{eq: Kaufman},  Proposition \ref{cor:main2} provides an example of a self-similar measure $\mu$ and a diffeomorphism $h:[0,1]\to[0,1]$ such that:   $h$ is not affine on $\operatorname{supp}(\mu)$ in a quantitative sense,  yet both $\mu$ and $h\mu$ are not Rajchman.  

\subsection{Background} \label{section: background}   Classical Fourier analysis of Bernoulli convolutions provides the
starting point for both of our constructions. Recall that
$\mu_\lambda$ denotes the Bernoulli convolution with parameter
$\lambda$, introduced before
Theorem~\ref{thm_main_non-conjugated}. Erd\H{o}s proved that if
$\lambda^{-1}\neq 2$ is a Pisot number, then $\mu_\lambda$ is not Rajchman;
by the law of pure type, this implies that
$\mu_\lambda\perp\operatorname{Leb}$ whenever $\lambda\in (0,1)$
(cf.~\cite{solomyak2004notes}). To date, these remain the only
parameters in the super-critical regime $(1/2,1)$ for which $\mu_\lambda$ is known to be singular. Conversely,
Salem proved that if $\lambda^{-1}$ is not Pisot, then $\mu_\lambda$ is
Rajchman. We adapt Salem's argument in the proof of
Theorem~\ref{thm_main-conjugated_poly-decay}, while an adaptation of
Erd\H{o}s' argument underlies the proof of
Theorem~\ref{thm_main_non-conjugated}; see
Section~\ref{Section: sketch} for more details. Determining precisely for which
$\lambda\in(1/2,1)$ the measure $\mu_\lambda$ is absolutely continuous
remains a major open problem, and this is closely related to the Fourier decay problem \cite{Shmkerin2014Abs}. We refer to the ICM surveys of Hochman
\cite{Hochman2018ICM} and Varj\'{u} \cite{varju2023self} for an overview
of progress on the regularity of Bernoulli convolutions and their
classical and modern connections with Fourier analysis.

Let us now discuss Theorem~\ref{thm_main_non-conjugated}. Jordan and
Sahlsten \cite{Sahl2016Jor} proved that certain invariant measures for
the highly nonlinear Gauss map have positive Fourier dimension.
Subsequent works of Bourgain and Dyatlov \cite{Bour2017dya} and Li
\cite{li2018fourier} introduced new methods, based  on
additive combinatorics and probability theory respectively. These allow  proving polynomial
Fourier decay for measures arising from dynamical, arithmetic, and
geometric settings, including many Patterson--Sullivan and Furstenberg
measures (see \cite{LequenSahlsten2026} for a more direct approach to PS measures). More recently, Algom, Rodriguez Hertz, and Wang
\cite{algom2023polynomial}, and independently Baker and Sahlsten
\cite{Baker2023Sahl}, combined these ideas with methods from smooth
dynamics to obtain polynomial Fourier decay in the more general setting of
self-conformal measures on the line; see also \cite{sahlsten2020fourier, algom2021decay}. In
particular, \cite{algom2023polynomial} shows that every self-conformal
measure associated to a $C^2([0,1])$ IFS
$\Phi$ has positive Fourier dimension, provided
that there is no $C^2$ diffeomorphism $g$ for which the IFS
$\Psi
=
\left\{
g\circ f_i\circ g^{-1}
:
f_i\in\Phi
\right\}$
is linear in the sense of \eqref{eq: linear}. The proof relies on a
suitable variant of Dolgopyat's method \cite{Dol1998annals}, whose
applicability is ensured by the failure of conjugacy to a linear
system. The paper \cite{algom2023polynomial} also contains the statement of the $C^\omega$ Fourier decay criteria (non $C^\omega$-conjugation to self-similar) discussed before Theorem \ref{thm_main_non-conjugated}, but its proof also relies on suitable variants of \eqref{eq: Kaufman} found in \cite{Algom2025wu, Baker2025banaji}, that we will soon discuss.

The methods just described all require at least $C^2$ regularity. An earlier result
of Algom, Rodriguez Hertz, and Wang \cite{algom2020decay} shows that,
for a $C^{1+\gamma}$ IFS with $\gamma>0$, all self-conformal measures
are Rajchman under the aperiodicity assumption
\[
\left\{
\log |f_i'(x_i)|
:
f_i(x_i)=x_i,\ f_i\in\Phi
\right\}
\text{ is not contained in a translate of a lattice.}
\]
Obtaining effective Fourier decay estimates under only
$C^{1+\gamma}$ regularity is a difficult problem; see the recent
work of Leclerc, Paukkonen, and Sahlsten
\cite{leclerc2025fourier} for progress in this direction. Our examples
lie at the opposite extreme: the derivatives of all the defining maps
are constant on the attractor, and thus the IFS is linear in the sense of \eqref{eq: linear}. 
Theorem~\ref{thm_main_non-conjugated} therefore shows, in particular,
that conjugacy to a linear IFS cannot be replaced by conjugacy to a
self-similar IFS in the Fourier decay criterion of
\cite{algom2023polynomial}.

Let us now discuss Theorem~\ref{thm_main-conjugated_poly-decay}.
Kaufman \cite{Kaufman1984ber} was the first to note the
phenomenon in \eqref{eq: Kaufman}. He proved it for the Bernoulli
convolution $\mu_\lambda$ with $\lambda\in(0,\frac{1}{2})$. Mosquera and Shmerkin
\cite{Shmerkin2018mos} later extended this result to all homogeneous,
(equicontractive) self-similar measures and, in particular, to all
Bernoulli convolutions. More recently, Algom, Chang, Meng Wu, and
Yu-Liang Wu \cite{Algom2025wu}, and independently Baker and Banaji
\cite{Baker2025banaji}, proved \eqref{eq: Kaufman} for all
self-similar measures. In fact, the result of \cite{Algom2025wu}
applies to maps $g\in C^{1+\gamma}$, where $\gamma>0$, under a
quantitative non-degeneracy assumption on $g'$. Roughly speaking,
this assumption requires $g'$ to be uniformly locally invertible on the support
of the measure, with a H\"older-continuous local inverse. Very
recently, Banaji and Yu \cite{Banaji2025yu,Banaji2026Qyu} obtained
both quantitative and qualitative higher-dimensional analogues of
\eqref{eq: Kaufman}; see also \cite{Mosquer2023aOlivo}.

With the exception of \cite{Banaji2025yu}, virtually all known proofs
of \eqref{eq: Kaufman} use  $L^2$-flattening: this means that for
every $\epsilon>0$ there exists an integer $p\geq 1$ such that
$\int_{B(0,R)}
\left|\widehat{\nu}(\xi)\right|^p\,d\xi
\lesssim_{\epsilon} R^{\epsilon}$
for all  $R>0$. The fact that self-conformal measures
on $\mathbb{R}$ satisfy such estimates was proved, in increasing
generality, by Kaufman \cite{Kaufman1984ber}, Tsujii
\cite{Tsujii2015self}, and Rossi and Shmerkin
\cite{Rossi2020Shmerkin}. Khalil \cite{khalil2023exponential}
developed a general criterion for $L^2$-flattening; see also
\cite{algom2025khalil,algom2026orponen}. A broad framework for
deducing Fourier decay from $L^2$-flattening  was recently
developed by Baker, Khalil, and Sahlsten
\cite{baker2024Kahlil}.

A separate line of work, closer to our method of a random construction than to the motivating phenomenon in \eqref{eq: Kaufman}, begins with Bluhm \cite{Bluhm1999fourier}. He showed that certain random perturbations of self-similar Cantor sets are almost surely Salem, through Fourier-decay estimates for associated random measures. In the homogeneous setting, the perturbation may be realized by a bi-Lipschitz map. Ekstr\"om \cite{Ekstrom2016random} later obtained a smooth version of this phenomenon at every finite regularity $C^r$, where $1\leq r<\infty$.

The $L^2$-flattening method,  as well as the more direct
methods of \cite{Algom2025wu}, cannot be applied to prove
Theorem~\ref{thm_main-conjugated_poly-decay}, since in our
construction $g'$ is constant on the support of $\mu_\lambda$. For a comparison with Ekstr\"om's work see Remark \ref{rmk: ekstorm}.

Finally, recent work of Baker and Banaji
\cite{baker2026banaji} complements the positive Fourier-decay results
of \cite{algom2023polynomial,Algom2025wu,Baker2023Sahl,Baker2025banaji}
in a different direction. They construct broad classes of
self-similar and self-conformal measures which are Rajchman but whose
Fourier transforms decay arbitrarily slowly, in settings not covered
by the preceding results. This phenomenon is already highly
non-trivial in the self-similar setting. In the self-conformal case,
however, the methods of \cite{baker2026banaji} require non-trivial
second-derivative. Thus, they are complementary to the
constructions in the present paper. We refer to Sahlsten's recent
survey \cite{sahlsten2025fourier} for further background on the
Fourier-decay problem.

\subsection{Pseudo-affine self-conformal measures and proof sketch} \label{Section: sketch}
All measures appearing in Theorems \ref{thm_main_non-conjugated} and \ref{thm_main-conjugated_poly-decay}  belong to a specific class of linear self-conformal measures that we now recall.

\begin{definition}
\label{defn_pseudo-affine}
Let $\Phi=\lbrace f_0,f_1\rbrace$ be a $C^s([0,1])$ IFS, $s\geq 1$.
\begin{enumerate}
    \item We say that $\Phi$ is \emph{hyperbolic} if $0<f_i'(x)<1$ for every $x\in[0,1]$, $i=0,1$, and
    \[
    0=f_0(0)<f_0(1)<f_1(0)<f_1(1)=1.
    \]

    \item We say that $\Phi$ is \emph{pseudo-affine with slope $\lambda$},
    or \emph{$\lambda$-pseudo-affine}, if
    \[
    f_i'(x)=\lambda
    \qquad\text{for every }x\in X_\Phi\text{ and }i=0,1.
    \]
    We say that $\Phi$ is \emph{pseudo-affine} if it is
    $\lambda$-pseudo-affine for some $\lambda>0$.
\end{enumerate}
\end{definition}

The origins of pseudo-affine IFS are in the following question of Hochman: Given $s\in [2,\infty]$, do there exist linear, in the sense of \eqref{eq: linear}, $C^s$-smooth IFS that are nonetheless not $C^s$-conjugated to a self-similar IFS?

In \cite{ABRHS_pseudo-Affine_IFS} we solved this problem, by introducing pseudo-affine IFS (that are clearly linear if $s\geq 2$), and showing that there exist such IFS that are not $C^1$-conjugated to self-similar, for every $s\in [1,\infty]$. In fact, we gave  a complete classification based on the notion of \emph{dynamical proportions}, which we now informally recall. It is rooted in Sullivan's scaling function \cite{Sullivan1987ratio}, and subsequent work of Bedford-Fisher \cite{BedfordFisher1997Ratio} and Bam\'on,  Moreira,  Plaza, and Vera \cite{Bamon1997gogo}. 
The dynamical proportions are a pair of positive functions on
$\{0,1\}^*$ which record the ratios between each gap and its
distinguished descendants. When a Cantor set has zero Lebesgue measure, then each pair of admissible (in an appropriate sense) dynamical proportions uniquely determine the Cantor set as a subset of the interval, and  geometric properties of the Cantor set are encoded in them. $\lambda$-Pseudo-affinity is characterized by the condition that the dynamical proportions converge uniformly to the constant value \(\lambda\) as the
length of the coding word tends to infinity. Moreover, the deviation from the limit  $\lambda$ encodes precise information about the regularity of the IFS. 
We  summarize  the main properties of pseudo-affine IFS that will be used in this paper in Section \ref{subsection_pseudo-aff_Cantor_sets}, this time rigorously and with full details. All of these were developed in \cite{ABRHS_pseudo-Affine_IFS}.

Let us now sketch the proof of Theorem \ref{thm_main_non-conjugated}. It is based upon  a classical argument of Erd\"os \cite{Erdos1939ber} showing that, if $\lambda^{-1}$ is Pisot then $\mu_\lambda$ is not Rajchman. Indeed, it is not hard to see that, since Fourier transforms convert convolutions to products,
\begin{equation}
\label{eq_convolutive_structure_3}
    \mu_\lambda=\ast_{k=0}^{\infty}\left(\frac{1}{2}\delta_{0}+\frac{1}{2}\delta_{\lambda^k(1-\lambda)}\right), \text{ and so } \widehat{\mu_\lambda}(\xi)=\prod_{n=0}^\infty\frac{1+e^{-2\pi i\lambda^n(1-\lambda)\xi}}{2}=e^{-\pi i\xi}\cdot\prod_{n=0}^\infty \cos(\pi\lambda^n(1-\lambda)\xi). 
\end{equation}
Let's see that $|\widehat{\mu_\lambda}(\xi_N)|>\kappa$, for every $N\geq 1$, for some constant $\kappa>0$ and  the divergent sequence $\xi_N=\lambda^{-N}(1-\lambda)^{-1}$. By direct evaluation we obtain
\begin{equation*}
    |\widehat{\mu_\lambda}(\xi_N)|=\prod_{n=0}^\infty |\cos(\pi\lambda^{n-N})|=\underbrace{\left(\prod_{k=1}^{N} |\cos(\pi\lambda^{-k})|\right)}_{A_N}\cdot\underbrace{\left(\prod_{n=0}^{\infty} |\cos(\pi\lambda^{n})|\right)}_{B}.  
\end{equation*}
From the one side, since \(0<\lambda<1/2\), we have 
$B=\prod_{n=1}^{\infty}\cos(\pi\lambda^n).$
Moreover,
$-\log\bigl(\cos(\pi\lambda^n)\bigr)
\lesssim \lambda^{2n}$
for all sufficiently large \(n\). Therefore
$\sum_{n=1}^{\infty}
-\log\bigl(\cos(\pi\lambda^n)\bigr)<\infty,$
and hence \(B>0\). Now, for studying the quantities $A_N$, observe that 
$$|\cos(\pi u)|=|\cos(\pi\Vert u\Vert)|, \text{ where } \Vert u\Vert=\mathrm{dist}(u,\Z), \text{ for every } u\in\R.$$ 
Since $\lambda^{-1}$ is a Pisot number the sequence $\epsilon_k=\Vert\lambda^{-k}\Vert$ converges to zero exponentially fast \cite[Chapter I, \S2, Theorem 1]{Salem1963}. Thus, by the same argument as above, the sequence $\log(A_N)$ is non-increasing and converges to some real number $\delta$. We conclude that $|\widehat{\mu_\lambda}(\xi_N)|\geq \kappa=e^\delta B$, for every $N\geq 1$, and hence $\mu_\lambda$ is not Rajchman. 

Now, to Theorem \ref{thm_main_non-conjugated}. Starting from a Pisot parameter $0<\lambda<1/2$, we construct a pseudo-affine IFS whose dynamical proportions agree with the constant value $\lambda$ at all sufficiently deep scales, but with a carefully chosen defect at a prescribed finite place.  Since this defect disappears at small scales, the resulting system can be shown to be $C^\infty$ and pseudo-affine with slope $\lambda$. On the other hand, the defect is visible to the dynamical proportions, and the conjugacy criterion for pseudo-affine Cantor sets, Theorem \ref{them_criterion_for_conjgation},
 shows that the resulting IFS is not $C^1$-conjugate to  self-similar. We let $\mu$ be the uniform Cantor-Lebesgue measure with respect to this IFS.

The measure $\mu$ is not an infinite convolution, and so the Fourier transform of $\mu$ no longer has the  product structure \eqref{eq_convolutive_structure_3}. Still, we are able to adapt Erd\"os' argument as follows.  Let
$\mu_n=\frac{1}{2^n}\sum_{w\in\{0,1\}^n}\delta_{f_w(0)}$
be the $n$-th discrete approximation of $\mu$. We split $\mu_n$ according to the last symbol of the word. Namely, for $i\in\{0,1\}$, set
\[
\nu_n^i = \frac{1}{2^{n-1}}
\sum_{u\in\{0,1\}^{n-1}}
\delta_{f_{ui}(0)} .
\]
Thus $\nu_n^i$ is the distribution of the $n$-th approximation conditioned on the last symbol being $i$, and
\[
\mu_n=\frac12\nu_n^0+\frac12\nu_n^1.
\]
Using our specific choice of the dynamical proportions, we show that there are constants \(A,B\), determined by the chosen defect, such that appending a \(0\) produces no additional translation, while appending a \(1\) produces a translation of size \(\lambda^{n-1}A\) from state \(0\), and of size \(\lambda^{n-1}B\) from state \(1\).
 Thus the conditional measures satisfy a two-state renewal relation of the form
\[
\begin{pmatrix}
\nu_n^0\\
\nu_n^1
\end{pmatrix}
=
\frac12
\begin{pmatrix}
\mathrm{Id} & \mathrm{Id}\\
(T_{\lambda^{n-1}A})_* & (T_{\lambda^{n-1}B})_*
\end{pmatrix}
\begin{pmatrix}
\nu_{n-1}^0\\
\nu_{n-1}^1
\end{pmatrix}, \, \text{ where }  T_t(x)=x+t.
\]
Taking Fourier transforms converts this renewal relation into the matrix recursion
\[
\begin{pmatrix}
\widehat{\nu}_n^0(\xi)\\
\widehat{\nu}_n^1(\xi)
\end{pmatrix}
=
M_n(\xi)
\begin{pmatrix}
\widehat{\nu}_{n-1}^0(\xi)\\
\widehat{\nu}_{n-1}^1(\xi)
\end{pmatrix},
\qquad
M_n(\xi)=
\frac12
\begin{pmatrix}
1 & 1\\
e^{-2\pi i\xi\lambda^{n-1}A} &
e^{-2\pi i\xi\lambda^{n-1}B}
\end{pmatrix}.
\]
Thus the Riesz product structure is replaced by a "matrix" Riesz product:
\[
\widehat{\mu}(\xi)
=
\lim_{N\to\infty}
\frac12(1,1)
M_N(\xi)M_{N-1}(\xi)\cdots M_2(\xi)
\begin{pmatrix}
\widehat{\nu}_1^0(\xi)\\
\widehat{\nu}_1^1(\xi)
\end{pmatrix}.
\]
This is the substitute for the missing convolution structure.

The proof of non-decay follows similar arithmetic considerations as Erdős'
argument. We evaluate the matrix product at frequencies
\(\xi_N=\lambda^{-N}\), up to a fixed normalization. For indices well below
\(N\), the phases are exponentially close to integers by the Pisot
property; for indices well above \(N\), they are exponentially small. The
remaining indices form a transition block of uniformly bounded length.
Our choice of the defect ensures that a corresponding limiting matrix
product exists, and has a non-zero coefficient. Comparing the finite products with
this limiting product then yields
$|\widehat{\mu}(\xi_N)|\geq c>0$
uniformly in \(N\).

Let us now sketch the proof of
Theorem~\ref{thm_main-conjugated_poly-decay}. Fix
$0<\lambda<1/2$, an integer $r\geq1$, and $0<\alpha\leq1$, and set
$s=r+\alpha$. On the probability space
\[
\Omega=[-1,1]^{\mathbb{N}_0},
\qquad
\mathbb{P}
=
\left(
\frac12\operatorname{Leb}|_{[-1,1]}
\right)^{\mathbb{N}_0},
\]
let $(Y_n)_{n\geq0}$ be the coordinate functions. For a sufficiently
small constant $c>0$, define
\[
d_n(\omega)
=
(1-\lambda)\lambda^n
\left(
1+c\lambda^{(s-1)n}Y_n(\omega)
\right),
\qquad
n\geq0,
\]
and normalize by
$\widetilde d_n(\omega)
=
\frac{d_n(\omega)}
{\sum_{m=0}^{\infty}d_m(\omega)}.$
The numbers $\widetilde d_n(\omega)$ are random perturbations of the
displacements $(1-\lambda)\lambda^n$ in the standard construction of
the Cantor set $X_{\Phi_\lambda}$. They determine a homogeneous (i.e. with equal length cylinders) Cantor
set $X_\omega\subset[0,1]$ by requiring that, inside every level-$n$
cylinder, the displacement between the left endpoints of its two
children is $\widetilde d_n(\omega)$.
The perturbations satisfy
$d_n(\omega)
=
(1-\lambda)\lambda^n
+
O\left(\lambda^{sn}\right).$
Consequently, the level-$n$ cylinder and gap lengths are comparable to
$\lambda^n$, and the dynamical proportions of $X_\omega$ have the form
$\lambda+\theta_n(\omega),
\,
|\theta_n(\omega)|
\lesssim_\lambda
\lambda^{(s-1)n}.$
The pseudo-affine classification and regularity theorems,
Theorems~\ref{thm_classification_pAffine_IFS} and
\ref{thm_deviation_from_constant_paffine}, therefore realize
$X_\omega$ as the attractor of a $C^{r,\alpha}$
$\lambda$-pseudo-affine IFS $\Phi_\omega$.

Similar arguments verify the conjugacy criterion in
Theorem~\ref{them_criterion_for_conjgation}. Hence the canonical address-preserving map extends
to a $C^{r,\alpha}$ diffeomorphism $h_\omega$ conjugating
$\Phi_\lambda:=\{x\mapsto \lambda x,\; x\mapsto \lambda x+1-\lambda\}$ to $\Phi_\omega$. Moreover, the derivative formula in
Theorem~\ref{them_criterion_for_conjgation} shows that
$h_\omega'(x)
=
\frac{1}{L_0(\omega)}
\text{ for every }x\in X_{\Phi_\lambda}.$
In particular, $h_\omega'$ is constant on
$\operatorname{supp}(\mu_\lambda)$.

The second step is to identify the stationary measure. The coding
map of $\Phi_\omega$ is given explicitly by
\[
\Theta_{\Phi_\omega}(a)
=
\sum_{n=0}^{\infty}
a_{n+1}\widetilde d_n(\omega),
\qquad
a=(a_n)_{n\geq1}\in\{0,1\}^{\mathbb{N}}.
\]
So, the uniform stationary measure
$\nu_\omega
:=
(h_\omega)_*\mu_\lambda$
is the law of
$\sum_{n=0}^{\infty}
\epsilon_n\widetilde d_n(\omega),$
where the $\epsilon_n$ are independent and uniformly distributed on
$\{0,1\}$. In particular, 
\[
\widehat{\nu_\omega}(\xi)
=
\prod_{n=0}^{\infty}
\frac{
1+e^{-2\pi i\xi\widetilde d_n(\omega)}
}{2}, \text{ and hence }
\left|
\widehat{\nu_\omega}(\xi)
\right|
=
\prod_{n=0}^{\infty}
\left|
\cos\bigl(
\pi\xi\widetilde d_n(\omega)
\bigr)
\right|.
\]
It is not hard to see that it suffices to prove the Fourier decay estimate for  $\eta_\omega \sim
\sum_{n=0}^{\infty}
\epsilon_nd_n(\omega).$
Indeed, there is a  random constant $L_0(\omega)$ such that
$\left|
\widehat{\eta_\omega}(\xi)
\right|
=
\prod_{n=0}^{\infty}
\left|
\cos\bigl(
\pi\xi d_n(\omega)
\bigr)
\right|,$ and
$\widehat{\nu_\omega}(\xi)
=
\widehat{\eta_\omega}
\left(
\frac{\xi}{L_0(\omega)}
\right).$
The advantage of working with $\widehat{\eta_\omega}$ is that the $n$-th factor in the last product now
depends only on the  random variable $Y_n$.

The Fourier-decay estimate is obtained by a Salem-type 
argument. The basic input is the following one-scale estimate. If
$Z$ is uniformly distributed on $[-1,1]$, then for every $p\geq1$,
$0<a\leq1/2$, and $au\geq1$,
\[
\mathbb{E}
\left|
\cos\bigl(\pi u(1+aZ)\bigr)
\right|^p
\leq
\rho_p,
\qquad
\rho_p
=
\sqrt{\frac{9}{2\pi p}}<1
\]
for $p$ sufficiently large. 
At a frequency $|\xi|\asymp\lambda^{-N}$, the factor corresponding to
the displacement $d_n$ has a random phase variation of size comparable
to
$|\xi|\lambda^{sn}.$
Consequently, there are at least
$\frac{N}{s}-O(1)$
indices $n$ for which this variation is at least one. The corresponding
factors depend on distinct random variables $Y_n$ and are therefore
independent. Applying the one-scale estimate at these effective scales
gives
\[
\mathbb{E}
\left|
\widehat{\eta_\omega}(\xi)
\right|^p
\leq
C_p\rho_p^{N/s},
\qquad
|\xi|\asymp\lambda^{-N}.
\]
Since $N$ is comparable to $\log|\xi|/\log(1/\lambda)$, this is a
polynomial moment estimate in $|\xi|$.

We then choose
$p
=
\frac{9e}{2\pi}\lambda^{-2s}$
and apply Markov's inequality on a sufficiently fine net in each
frequency shell $|\xi|\asymp\lambda^{-N}$. The preceding moment bound
makes the probabilities of failure summable, so the Borel--Cantelli
lemma yields almost-sure polynomial decay at all net points. A uniform
Lipschitz bound for the Fourier transforms extends this estimate from
the nets to all frequencies. After restoring the bounded normalization
factor $L_0(\omega)$, we obtain, for almost every $\omega$,
\[
\left|
\widehat{\nu_\omega}(\xi)
\right|
\leq
C_\omega|\xi|^{-\tau},
\qquad
|\xi|\geq1,
\]
where one may take
$\tau
=
\frac{2\pi\lambda^{2s}}
{(18e+\pi)s\log(1/\lambda)}.$
Choosing one such realization $\omega$ and setting $g=h_\omega$
proves Theorem~\ref{thm_main-conjugated_poly-decay}.

Let us finally compare our construction with that of Ekstr\"om
\cite{Ekstrom2016random}. Both arguments introduce independent random
perturbations and obtain almost-sure Fourier decay from high-moment
estimates, but the randomness is organized differently. Ekstr\"om
starts with a general compact set $E$ and assigns an independent
random increment to each bounded component of $\mathbb{R}\setminus E$,
thereby randomly enlarging the individual complementary intervals.
He extends the resulting map smoothly across the gaps by means of
suitable bump functions. The added perturbation is flat to the
available order on $E$; in particular, the derivative of the
resulting diffeomorphism is equal to $1$ on $E$. His Fourier argument
expands high moments of the Fourier transform and uses the abundance
of sufficiently large complementary intervals separating the points
appearing in this expansion. Our construction is more rigid: there
is one random variable at each level, and the same displacement is
used in every cylinder of that level. This level-homogeneous
organization preserves an explicit self-conformal structure and
produces the exact Riesz-product formula above. After removing the
single normalization factor $L_0(\omega)$, the Fourier factors
themselves become independent, allowing the decay estimate to be
obtained directly from a one-scale cosine moment bound. Thus
Ekstr\"om's construction applies to a substantially broader class of
sets and measures, whereas ours retains and identifies the induced
pseudo-affine dynamics, and can also be extended to  $C^\infty$ - see Proposition \ref{prop_smooth_stretched_log_decay}.

\subsection{Acknowledgments}

The authors thank Simon Baker, Amlan Banaji, Gaétan Leclerc, and Meng Wu for helpful comments and discussions. A.A. is supported by the Israel Science Foundation, Grant No.~392/25, NSF--BSF Grant No.~2024692, and Grant No.~2022034 from the United States--Israel Binational Science Foundation (BSF), Jerusalem, Israel. F. RH. is partially supported by NSF Grant No.~2453688 and by the Anatole Katok Chair in Mathematics. M.S. is partially supported by UdelaR-CSIC Grant 149/348 - Geometry and group actions.

%---------------------------------------------------------------------
%---------------------------------------------------------------------

\section{Preliminaries}

In this section we collect notations, results, and methods from
\cite{ABRHS_pseudo-Affine_IFS} that will be used in the proofs of
Theorems  
\ref{thm_main-conjugated_poly-decay} and \ref{thm_main_non-conjugated}.  
%Recall that pseudo-affine IFSs was defined in Definition \ref{defn_pseudo-affine}, and  self-conformal measures was defined in \eqref{eq: self-similar measure}.

\subsection{Binary IFS, coding, and stationary measures}
\label{subsection_2.1}

Throughout, $\Phi=\{ f_0,f_1\}$ will denote an IFS generated by two $C^s$-smooth maps $f_i:[0,1]\to [0,1]$ verifying 
$$0=f_0(1)<f_0(1)<f_1(0)<f_1(1)=1,$$
where $s\in [1,\infty]$.  Recall that the IFS is said to be \emph{hyperbolic} if in addition the generators have derivatives $0<f_i'(t)<1$ for every $t\in[0,1]$, $i=1,2$. This will be our standard assumption throughout the text. 
We denote by $X=X_\Phi$  its
attractor, which is homeomorphic to a Cantor set.

Let
\[
\W=\bigcup_{n\geq 0}\{0,1\}^n
\]
be the set of finite words over the alphabet $\{0,1\}$, including the
empty word $e$. For a word $w=w_1\cdots w_n$, write $|w|=n$ and set
\[
F_w=f_{w_1}\circ\cdots\circ f_{w_n},
\qquad
F_e=\operatorname{Id}.
\]
For every $w\in\W$, let
\[
K_w=F_w([0,1]),
\qquad
X_w=F_w(X)=X\cap K_w.
\]
The sets $X_w$ are called the \emph{cylinders} of $X$. The \emph{gaps} of $X$ are the connected components of $[0,1]\setminus X$ and will be labeled as follows: The \emph{central gap} of $X$ is the interval 
\[
I=[0,1]\setminus(K_0\cup K_1),
\]
and for $w\in\W$, define the \emph{$w$-gap} of $X$ to be
\[
I_w=F_w(I)=K_w\setminus(K_{w0}\cup K_{w1}).
\]
The family of all gaps of $X$ is thus $\{I_w:w\in\W\}$.

Writing
$\Sigma_2^+=\{0,1\}^{\mathbb N},$ we define a 
 coding map $\Theta:\Sigma_2^+\to X$  by
\begin{equation} \label{eq: coding}
\Theta_\Phi(a)=\Theta(a)=x_a
:=
\bigcap_{n\geq 1}K_{a_1\cdots a_n},
\qquad
a=(a_n)_{n\geq1}.
\end{equation}
It satisfies
$f_i(x_a)=x_{ia},
\,i=0,1.$

We say that two IFSs $\Phi=\{ f_0,f_1\}$ and
$\Psi=\{ g_0,g_1\}$ are \emph{$C^s$-conjugate on their
attractors} if there exists a $C^s$ diffeomorphism
$h:[0,1]\to[0,1]$ such that
\[
h\circ f_i(x)=g_i\circ h(x),
\qquad
x\in X_\Phi,\quad i=0,1.
\]
Notice that this implies that $h(X_\Phi)=X_\Psi$
Equivalently, they are conjugate if
$h\circ F_w(x)=G_w\circ h(x)$, 
for every $w\in\W$ and $x\in X_\Phi$.

For $0<\lambda<1/2$, we write the IFS generating the Bernoulli convolution $\mu_\lambda$ by
\[
\Phi_\lambda=\{ e_0,e_1\},
\qquad
e_0(x)=\lambda x,
\qquad
e_1(x)=\lambda x+1-\lambda.
\]

Finally, recall the definition of self-conformal measures at expression \eqref{eq: self-similar measure}. Let
$\nu= (\frac{1}{2},\,\frac{1}{2})^\mathbb{N}$
be the uniform Bernoulli measure on $\Sigma_2^+$. The corresponding
stationary measure on $X$ is the self-conformal (Cantor-Lebesgue) measure
$\mu=\Theta_*\nu.$
Equivalently, $\mu$ is the unique probability measure satisfying the \emph{transfer equation}
$\mu=\frac12(f_0)_*\mu+\frac12(f_1)_*\mu.$
Notice that the discrete measures
\[
\mu_n
=
\frac{1}{2^n}
\sum_{w\in\{0,1\}^n}\delta_{F_w(0)}
\]
converge weakly to $\mu$.

\subsection{Dynamical proportions and pseudo-affine IFSs}
\label{subsection_intro_dynamical_proportions}
\label{subsection_pseudo-aff_Cantor_sets}

We continue to work with an hyperbolic IFS
$\Phi=\{ f_0,f_1\}$. Its \emph{dynamical proportions} are the
functions
\begin{equation}
\label{eq_dynamical_prop_defn}
\lambda_i:\W\to(0,\infty),
\qquad
\lambda_i(w)=\frac{|I_{iw}|}{|I_w|},
\qquad i=0,1.
\end{equation}
We define the associated multiplicative cocycle  by
\begin{equation}
\label{eq_cocycle}
\Psi(e)=1,
\qquad
\Psi(w_1\cdots w_n)
=
\prod_{j=1}^n
\lambda_{w_j}(w_{j+1}\cdots w_n).
\end{equation}
By construction,
$|I_w|=\Psi(w)|I|$
for every $w\in\W$, where $I=I_e$ is the central gap.

The following elementary (re)construction  will be useful when
computing cylinder lengths.

\begin{proposition}[{\cite[Lemma 8]{ABRHS_pseudo-Affine_IFS}}]
\label{prop_proportion_functions_and_Cantor_set}
Let $\lambda_0,\lambda_1:\W\to(0,\infty)$ and let $\Psi$ be the
cocycle defined by \eqref{eq_cocycle}. Suppose that
\[
1<
\sum_{w\in\W}\Psi(w)
<
\infty.
\]
Then the  central
gap of $X=X_\Phi$ has length
\[
|I|
=
\left(
\sum_{w\in\W}\Psi(w)
\right)^{-1},
\]
and, for every $w\in\W$,
\begin{equation}
\label{eq_cylinder_length}
\operatorname{diam}(X_w)
=
\sum_{u\in\W}|I_{wu}|
=
|I|\sum_{u\in\W}\Psi(wu).
\end{equation}
\end{proposition}

We now recall the three results from
\cite{ABRHS_pseudo-Affine_IFS} that will be used in the sequel. Recall
that pseudo-affine IFSs were defined in
Definition \ref{defn_pseudo-affine}.

\begin{theorem}[{
\cite[Propositions 12, 13, 15]{ABRHS_pseudo-Affine_IFS}}]
\label{thm_deviation_from_constant_paffine}
Let 
$0<\lambda<1/2$. Then our hyperbolic IFS $\Phi$ is pseudo-affine with slope $\lambda$ if and
only if its dynamical proportions may be written as
\[
\lambda_i(w)=\lambda+\theta_i(w),
\, i=0,1, \text{ where }\sup_{|w|=n}|\theta_i(w)|
\longrightarrow 0
\qquad\text{as }n\to\infty.
\]

Moreover, let $s=r+\alpha>1$, where $r\geq1$ is an integer and
    $0<\alpha\leq1$. Then $\Phi$ is of class $C^{r,\alpha}$ if and only
    if there exists $C>0$ such that
    \[
    |\theta_i(w)|
    \leq
    C\Psi(w)^{s-1}
    \]
    for every $w\in\W$ and $i=0,1$. In particular, it is $C^\infty$ if and only if, for every
    integer $m\geq1$, there exists $C_m>0$ such that   
$|\theta_i(w)|
    \leq
    C_m\Psi(w)^m$
    for every $w\in\W$ and $i=0,1$.

\end{theorem}

The preceding theorem characterizes the regularity of a pseudo-affine
system once its dynamical proportions are known. The following result
allows us to construct such systems directly from prescribed
proportions.

\begin{theorem}
[{\cite[Proposition 15]{ABRHS_pseudo-Affine_IFS}}]
\label{thm_classification_pAffine_IFS}
Let $0<\lambda<1/2$ and let
$s=r+\alpha>1$, where $r\geq1$ is an integer and
$0<\alpha\leq1$. Suppose that
$\theta_i:\W\to\mathbb R,
\, i=0,1,$
satisfy the following conditions:
\begin{enumerate}[(a)]
    \item
$\sup_{|w|=n}|\theta_i(w)|
    \rightarrow0
    \,\text{ as }n\to\infty;$

    \item $0<\lambda+\theta_i(w)<1$
    for every $w\in\W$ and $i=0,1$;

    \item There exists some $C>0$ such that 
    $|\theta_i(w)|
    \leq
    C\Psi(w)^{s-1}$
    for every $w\in\W$ and $i=0,1$, where $\Psi$ is the cocycle
    associated to
    $\lambda_i(w)=\lambda+\theta_i(w).$
\end{enumerate}
Then there exists a hyperbolic pseudo-affine IFS of class
$C^{r,\alpha}$ and slope $\lambda$ whose dynamical proportions are
$\lambda_i(w)=\lambda+\theta_i(w).$

Moreover, if  for every integer $m\geq1$, there exists $C_m>0$ such
    that
    $|\theta_i(w)|
    \leq
    C_m\Psi(w)^m$
    for every $w\in\W$ and $i=0,1$,
then the resulting IFS may be chosen of class $C^\infty$.
\end{theorem}

Finally, we recall a conjugacy criterion for pseudo-affine IFSs.
Let $\Phi$ and $\Psi$ be hyperbolic IFSs with attractors
$X_\Phi$ and $X_\Psi$, respectively, and let
$\Theta_\Phi$ and $\Theta_\Psi$ be their coding maps, defined as in
\eqref{eq: coding}. There is a canonical  address-preserving homeomorphism from
$X_\Psi$ to $X_\Phi$, defined by
\[
h_{\Psi\to\Phi}
:=
\Theta_\Phi\circ\Theta_\Psi^{-1}.
\]
Equivalently,
$h_{\Psi\to\Phi}
\bigl(\Theta_\Psi(a)\bigr)
=
\Theta_\Phi(a)
\,
\text{ for every }a\in\Sigma_2^+.$
This map conjugates the actions of $\Psi$ and $\Phi$ on their
attractors.

\begin{theorem}[{
\cite[Theorem 20]{ABRHS_pseudo-Affine_IFS}}]
\label{them_criterion_for_conjgation}
Let $\Phi_1=\{f_0,f_1\}$ and $\Phi_2=\{g_0,g_1\}$ be hyperbolic
pseudo-affine IFSs of class $C^s$, $s\geq1$, with the same slope
$\lambda>0$. Let  $\Psi_{\Phi_1}$ and $\Psi_{\Phi_2}$ denote their respective dynamical cocycles. Then the canonical map
$h_{\Phi_2\to\Phi_1}$
extends to a $C^s$ diffeomorphism of $[0,1]$ if and only if, \text{for every }
$a=(a_n)_{n\geq1}\in\Sigma_2^+,$
\begin{equation} \label{eq_cocycle_conjugation}
\chi_{\Phi_1,\Phi_2}(a)
=
\lim_{n\to\infty}
\frac{\Psi_{\Phi_1}(a_1\cdots a_n)}
     {\Psi_{\Phi_2}(a_1\cdots a_n)} \text{ exists and defines a continuous function }
\chi_{\Phi_1,\Phi_2}:\Sigma_2^+\to(0,\infty).
\end{equation}
In this case, if $I_{\Phi_1}$ and $I_{\Phi_2}$ denote their central gaps, then
\[
h_{\Phi_2\to\Phi_1}'
\bigl(\Theta_{\Phi_2}(a)\bigr)
=
\frac{|I_{\Phi_1}|}{|I_{\Phi_2}|}
\chi_{\Phi_1,\Phi_2}(a)
\]
for every $a\in\Sigma_2^+$.
\end{theorem}

%---------------------------------------------------------------------
\section{On the proof of Theorem \ref{thm_main_non-conjugated}}

In  this section we prove  the existence of a hyperbolic IFS $\Phi=\{ f_0, f_1\}$  on $[0,1]$, of class $C^\infty$,  such that: 
\begin{enumerate}
    \item $\Phi$ is pseudo-affine,
    \item $\Phi$ is not $C^1$-conjugated to a self-similar IFS,
    \item Its Cantor-Lebesgue ($(1/2,1/2)$-Bernoulli) measure  is not  Rajchman. 
\end{enumerate}

The proof has three steps. Pick a slope
$0<\lambda<1/2$ such that $\lambda^{-1}$ is a large integer. First, we  use
the theory of dynamical proportions from
Section \ref{subsection_pseudo-aff_Cantor_sets} to construct a
$C^\infty$ $\lambda$-pseudo-affine IFS which is not $C^1$-conjugate to 
self-similar. Then, we study its uniform
$(1/2,1/2)$-Bernoulli measure $\mu$. Although $\mu$ does not have the structure of a Bernoulli convolution, its
finite approximations do satisfy a certain two-state recursion. We are thus able to show that $\widehat{\mu}$ has a kind of
matrix-type  product structure. Finally, we use this
structure and the Pisot property to execute an Erd\H{o}s-type argument, 
showing that $\widehat{\mu}$ does not tend to zero at infinity.

\subsection{Construction of the IFS} \label{Section: contruction of IFS non conj}
Let $0<\lambda<\frac{1}{2}$, and let $\varepsilon>0$ such that 
$$0<\kappa=\lambda+\varepsilon<1/2.$$ 
Later we will choose $\lambda$ so that $\lambda^{-1}$ is a large integer, and impose an additional arithmetic condition on $\varepsilon$.

Define as in Section \ref{subsection_intro_dynamical_proportions}, a pair of functions $\lambda_i:\W\to (0,+\infty)$, $i=0,1$, via:
\begin{align}
\label{eq_proportions_thm_A}
    & \lambda_1(e)=\kappa=\lambda+\varepsilon,\\
    & \lambda_i(w)=\lambda,\ \text{whenever}\ i= 0\ \text{or}\ i=1\ \text{and}\ w\neq e. \nonumber
\end{align}
Equivalently, by writing $\lambda_i=\lambda+\theta_i$, we can set: 
\begin{align}
\label{eq_proportions_deviation_thm_A}
    & \theta_1(e)=\varepsilon,\\
    & \theta_i(w)=0,\ \text{whenever}\ i= 0\ \text{or}\ i=1\ \text{and}\ w\neq e. \nonumber
\end{align}

\begin{proposition}
\label{prop_construction_IFS_sec_3}
    There exists a $\lambda$-pseudo-affine hyperbolic IFS $\Phi=\{ f_0,f_1 \}$  on $[0,1]$ such that:
    \begin{enumerate}
        \item $X=X_\Phi$ has the dynamical proportion $(\lambda_0,\lambda_1)$ from  \eqref{eq_proportions_thm_A}, in the sense defined in Section \ref{subsection_intro_dynamical_proportions};

        \item $\Phi$ is of class $C^\infty$; and,

        \item $\Phi$ is not $C^1$-conjugated to a self-similar IFS.
    \end{enumerate}
\end{proposition}

\begin{proof}
The functions $\theta_i$ defined in
\eqref{eq_proportions_deviation_thm_A} satisfy the admissibility
conditions of Theorem
\ref{thm_classification_pAffine_IFS}. Indeed, they converge uniformly
to zero as $|w|\to\infty$, and
$0<\lambda+\theta_i(w)<\frac12$
for every $w\in\W$ and $i=0,1$. Moreover, for every integer $m\geq1$,
$|\theta_i(w)|
\leq
\varepsilon\Psi(w)^m,$
since $\theta_i(w)$ vanishes unless $(i,w)=(1,e)$, and
$\Psi(e)=1$. Theorem
\ref{thm_classification_pAffine_IFS} therefore yields a
$C^\infty$ hyperbolic IFS $\Phi$ which is pseudo-affine with slope
$\lambda$, and has the prescribed dynamical proportions.

It remains to prove that $\Phi$ is not $C^1$-conjugate to a
self-similar IFS. Suppose, to the contrary, that $\Phi$ is
$C^1$-conjugate via a map $h$ to a self-similar IFS
$\Psi=\{g_0,g_1\}.$
Let $x_i$ be the fixed point of
$f_i$. Since $x_0=0$ and $x_1=1$, the points
$p_i:=h(x_i),
\, i=0,1,$
are the fixed points of $g_i$. In particular, $p_0\neq p_1$. Differentiating the conjugacy relation at $x_i$ gives
\[
h'(x_i)f_i'(x_i)
=
g_i'(p_i)h'(x_i).
\]
Since $f_i'(x_i)=\lambda$ and $h'(x_i)\neq0$, it follows that
$g_i'=\lambda,
\, i=0,1.$
As $g_i$ is affine and fixes $p_i$, we therefore have
$g_i(x)=\lambda x+(1-\lambda)p_i.$

Now define the affine map
$L(x):=\frac{x-p_0}{p_1-p_0}.$
Then
$L(p_0)=0,
\,
L(p_1)=1,$
and a direct calculation gives
\[
L\circ g_0\circ L^{-1}(x)=\lambda x,
\qquad
L\circ g_1\circ L^{-1}(x)
=
\lambda x+1-\lambda.
\]
Thus $L\circ h$ is a $C^1$ conjugacy between $\Phi$ and
$\Phi_\lambda
=
\{\lambda x,\lambda x+1-\lambda\}.$
Consequently, it is enough to show that $\Phi$ is not
$C^1$-conjugate to $\Phi_\lambda$.

We now apply Theorem
\ref{them_criterion_for_conjgation}. For
$a=(a_n)_{n\geq1}\in\Sigma_2^+$, set
\[
R_n(a)
:=
\frac{\Psi_\Phi(a_1\cdots a_n)}
     {\Psi_{\Phi_\lambda}(a_1\cdots a_n)}
=
\frac{\Psi_\Phi(a_1\cdots a_n)}
     {\lambda^n}.
\]
Since the only exceptional dynamical proportion is
$\lambda_1(e)=\kappa$, we have
\[
R_n(a)
=
\begin{cases}
1,
& a_n=0,\\[2mm]
\dfrac{\kappa}{\lambda},
& a_n=1.
\end{cases}
\]
Taking
$a=(01)^\infty,$
we obtain
$R_{2k}(a)=\frac{\kappa}{\lambda},
\,
R_{2k+1}(a)=1.$
Thus $(R_n(a))_{n\geq1}$ does not converge. By Theorem
\ref{them_criterion_for_conjgation}, $\Phi$ is not
$C^1$-conjugate to $\Phi_\lambda$, and hence is not
$C^1$-conjugate to any self-similar IFS.
\end{proof}
%------------------------------------------------------------------
\subsection{Description of the stationary measure}
\label{subsec_description_stationary_measure}

Let $(a_n)_{n\geq1}$ be independent random variables, uniformly
distributed on $\{0,1\}$, and set
\[
x_0=0,
\qquad
x_n=F_{a_1\cdots a_n}(0),
\quad n\geq1,
\]
where $\Phi=\{f_0,f_1\}$ is the IFS constructed in
Proposition \ref{prop_construction_IFS_sec_3}. Let $\mu_n$ be the law
of $x_n$. As noted in Section \ref{subsection_2.1},
$\mu_n\xrightarrow[n\to\infty]{w^*}\mu,$
where $\mu$ is the uniform self-conformal measure associated to
$\Phi$.

For a finite word $w$, write
\[
z_w:=F_w(0)
\]
for the left endpoint of the cylinder $X_w$, and write
\[
|X_w|:=\operatorname{diam}(X_w).
\]
Since the right endpoint of $X_{w1}$ coincides with the right endpoint
of $X_w$, we have
\[
z_{w0}=z_w,
\qquad
z_{w1}=z_w+|X_w|-|X_{w1}|.
\]
Thus the increments of $x_n$ are determined by the
diameters of the cylinders. Indeed, Figure \ref{fig:Cantor_process} illustrates the three possible types
of increments: the first-level displacement and the two displacements
that arise at later levels according to the last symbol of the
current word.

\begin{figure}[t]
    \centering
    \includegraphics[width=1\linewidth]{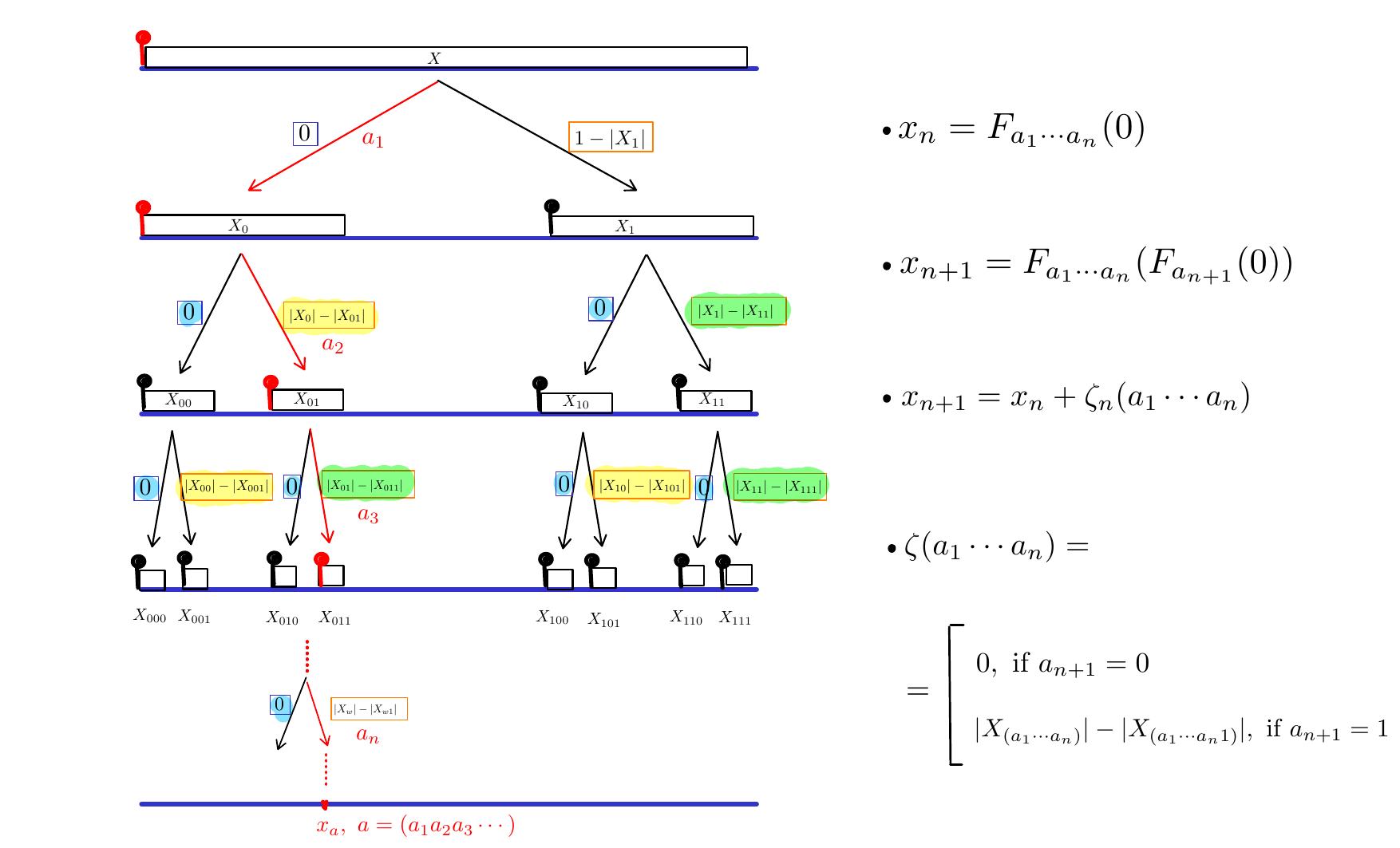}
    \caption{The increments in the random orbit $(x_n)_{n\geq0}$.}
    \label{fig:Cantor_process}
\end{figure}

Recall the choice of parameters $\lambda, \kappa$ from Section \ref{Section: contruction of IFS non conj}.
\begin{proposition}
\label{prop_cylinder_sizes_non-conjugated_affine}
We have
\[
|X_0|=\lambda,\quad \text{ and }
\beta:=|X_1|
=
\frac{\lambda^2+\kappa(1-\lambda)}
     {1-\lambda+\kappa}.
\]
More generally, if $w=w_1\cdots w_n$ has positive length, then
\[
|X_w|
=
\begin{cases}
\lambda^n,
& w_n=0,\\[2mm]
\lambda^{n-1}\beta,
& w_n=1.
\end{cases}
\]
Consequently, setting
$A:=
\lambda(1-\beta),
\,
B:=
(1-\lambda)\beta,$
we have
\[
|X_w|-|X_{w1}|
=
\begin{cases}
\lambda^{n-1}A,
& w_n=0,\\[2mm]
\lambda^{n-1}B,
& w_n=1.
\end{cases}
\]
\end{proposition}

\begin{proof}
For a nonempty word $w$, the definition of the dynamical proportions
in \eqref{eq_proportions_thm_A} gives
\[
\Psi(w)
=
\begin{cases}
\lambda^{|w|},
& \text{if the last symbol of $w$ is $0$},\\[2mm]
\kappa\lambda^{|w|-1},
& \text{if the last symbol of $w$ is $1$}.
\end{cases}
\]
Therefore,
$\sum_{w\in\W}\Psi(w)
=
1+
\sum_{n=1}^{\infty}
2^{n-1}\bigl(\lambda^n+\kappa\lambda^{n-1}\bigr)
=
1+\frac{\lambda+\kappa}{1-2\lambda}.$
Proposition \ref{prop_proportion_functions_and_Cantor_set} now gives the size of the central gap $I$ of $X$ as
\[
c:=|I|
=
\left(\sum_{w\in\W}\Psi(w)\right)^{-1}
=
\frac{1-2\lambda}{1-\lambda+\kappa}.
\]
Since
$\Psi(0u)=\lambda\Psi(u)
\,
\text{ for every }u\in\W,$
formula \eqref{eq_cylinder_length} gives
\[
|X_0|
=
c\sum_{u\in\W}\Psi(0u)
=
c\lambda\sum_{u\in\W}\Psi(u)
=
\lambda.
\]
Similarly,
$\Psi(1)=\kappa,
\,
\Psi(1u)=\lambda\Psi(u)
\,\text{ for }u\neq e,$
and hence
\[
|X_1|
=
c\left(
\kappa+\lambda
\sum_{u\in\W\setminus\{e\}}\Psi(u)
\right)
=
\lambda+(\kappa-\lambda)c
=
\beta.
\]

The same calculation after a word $w$ of length $n$ gives
\[
|X_w|
=
\lambda^n
\quad\text{if }w_n=0, \text{ and } \,
|X_w|
=
\lambda^{n-1}\beta
\quad\text{if }w_n=1.
\]
Since $w1$ has length $n+1$ and ends in $1$,
\[
|X_{w1}|=\lambda^n\beta.
\]
Subtracting this from the two possible values of $|X_w|$ proves the
last assertion.
\end{proof}

It follows that
$x_1=(1-\beta)\1_{\{a_1=1\}},$
and, for every $n\geq2$,
\[
x_n-x_{n-1}
=
\begin{cases}
0,
& a_n=0,\\[1mm]
\lambda^{n-2}A,
& a_n=1,\ a_{n-1}=0,\\[1mm]
\lambda^{n-2}B,
& a_n=1,\ a_{n-1}=1.
\end{cases}
\]
Equivalently,
\begin{equation}
\label{eq_sequence_x_n}
x_n
=
(1-\beta)\1_{\{a_1=1\}}
+
\sum_{k=2}^n
\lambda^{k-2}\1_{\{a_k=1\}}
\left(
A\1_{\{a_{k-1}=0\}}
+
B\1_{\{a_{k-1}=1\}}
\right).
\end{equation}

For $i\in\{0,1\}$ and $n\geq1$, let
\[
\nu_n^i
:=
\mathcal{L}(x_n\mid a_n=i)
\]
be the conditional law of $x_n$ given its last symbol. Since $a_n$ is
uniform on $\{0,1\}$,
\[
\mu_n
=
\frac12\nu_n^0+\frac12\nu_n^1.
\]
At the first level,
\[
\nu_1^0=\delta_0,
\qquad
\nu_1^1=\delta_{1-\beta}.
\]

For $t\in\mathbb R$, let $T_t(x)=x+t$.

\begin{proposition}
\label{prop_conditional_recursion}
For every $n\geq2$, the pair $(\nu_n^0,\nu_n^1)$ satisfies
\begin{align}
\label{eq_convolutive_structure_non-rajchman}
\nu_n^0
&=
\frac12\nu_{n-1}^0+\frac12\nu_{n-1}^1,
\\
\nu_n^1
&=
\frac12
(T_{\lambda^{n-2}A})_*\nu_{n-1}^0
+
\frac12
(T_{\lambda^{n-2}B})_*\nu_{n-1}^1.
\nonumber
\end{align}
\end{proposition}

\begin{proof}
If $a_n=0$, then $x_n=x_{n-1}$. Since $a_n$ is independent of
$(a_1,\ldots,a_{n-1})$,
\[
\nu_n^0
=
\mathcal{L}(x_{n-1})
=
\mu_{n-1}
=
\frac12\nu_{n-1}^0+\frac12\nu_{n-1}^1.
\]

If $a_n=1$, then Proposition
\ref{prop_cylinder_sizes_non-conjugated_affine} gives
\[
x_n
=
\begin{cases}
x_{n-1}+\lambda^{n-2}A,
& a_{n-1}=0,\\[1mm]
x_{n-1}+\lambda^{n-2}B,
& a_{n-1}=1.
\end{cases}
\]
Conditioning on $a_{n-1}$, which remains uniformly distributed after
conditioning on $a_n=1$, gives the second identity.
\end{proof}

Let $\mathcal M(\mathbb R)$ denote the space of finite signed Borel
measures on $\mathbb R$. For each $n\geq1$, define a linear operator
$\Phi_n:\mathcal M(\mathbb R)^2\to\mathcal M(\mathbb R)^2$
by\begin{equation}
\label{eq_operators_Phi_k-defn}
\Phi_1
=
\frac12
\begin{pmatrix}
\operatorname{Id} & \operatorname{Id}\\
(T_{1-\beta})_* & (T_{1-\beta})_*
\end{pmatrix},
\qquad
\Phi_n
=
\frac12
\begin{pmatrix}
\operatorname{Id} & \operatorname{Id}\\
(T_{\lambda^{n-2}A})_* &
(T_{\lambda^{n-2}B})_*
\end{pmatrix},
\quad n\geq2.
\end{equation}
Then, with the convention that the operators act from right to left,
\begin{equation}
\label{eq_convolutive_structure_non-rajchman_2}
\begin{pmatrix}
\nu_n^0\\
\nu_n^1
\end{pmatrix}
=
\Phi_n\Phi_{n-1}\cdots\Phi_1
\begin{pmatrix}
\delta_0\\
\delta_0
\end{pmatrix},
\end{equation}

We represent the operators $\Phi_n$ by matrices with entries in
$\mathcal M(\mathbb R)$. For
$Q=(q_{ij})_{i,j=0,1}\in M_2(\mathcal M(\mathbb R))$
and
$\eta=
\begin{pmatrix}
\eta_0\\
\eta_1
\end{pmatrix}
\in\mathcal M(\mathbb R)^2,$
define
\[
Q*\eta
:=
\begin{pmatrix}
q_{00}*\eta_0+q_{01}*\eta_1\\
q_{10}*\eta_0+q_{11}*\eta_1
\end{pmatrix}.
\]
For $Q,R\in M_2(\mathcal M(\mathbb R))$, we thus define their convolution
product by
\[
(Q*R)_{ij}
:=
\sum_{k=0}^1 q_{ik}*r_{kj}.
\]
Set
\begin{equation}
\label{eq_matrix_Q_k}
Q_1
=
\frac12
\begin{pmatrix}
\delta_0 & \delta_0\\
\delta_{1-\beta} & \delta_{1-\beta}
\end{pmatrix},
\qquad
Q_n
=
\frac12
\begin{pmatrix}
\delta_0 & \delta_0\\
\delta_{\lambda^{n-2}A} &
\delta_{\lambda^{n-2}B}
\end{pmatrix},
\quad n\geq2.
\end{equation}
Thus, 
$\Phi_n(\eta)=Q_n*\eta.$
Writing
\begin{equation}
\label{eq_matrices_P_n}
P_n
:=
Q_n*Q_{n-1}*\cdots*Q_1,
\end{equation}
we have
$\begin{pmatrix}
\nu_n^0\\
\nu_n^1
\end{pmatrix}
=
P_n*
\begin{pmatrix}
\delta_0\\
\delta_0
\end{pmatrix}.$

\begin{proposition}
\label{prop_limit_operator}
The matrices $P_n$ converge entrywise in the weak$^*$ topology to
\begin{equation}
\label{eq_matrix_Q}
Q
=
\frac12
\begin{pmatrix}
\mu & \mu\\
\mu & \mu
\end{pmatrix}.
\end{equation}
Equivalently, if $\eta_0,\eta_1$ are compactly supported finite signed
measures and
$\overline{\eta}
:=
\frac12(\eta_0+\eta_1),$
$$\Phi_n\Phi_{n-1}\cdots\Phi_1
\begin{pmatrix}
\eta_0\\
\eta_1
\end{pmatrix}
\xrightarrow[n\to\infty]{w^*}
\begin{pmatrix}
\mu*\overline{\eta}\\
\mu*\overline{\eta}
\end{pmatrix}.$$
In particular,
$\nu_n^0\xrightarrow[n\to\infty]{w^*}\mu,
\,
\nu_n^1\xrightarrow[n\to\infty]{w^*}\mu.$
\end{proposition}

\begin{proof}
We first prove that
$\nu_n^0\xrightarrow[n\to\infty]{w^*}\mu$ and
$\nu_n^1\xrightarrow[n\to\infty]{w^*}\mu.$ Since $x_n=x_{n-1}$ conditional on the event $\{a_n=0\}$, and $a_n$ is
independent of $(a_1,\ldots,a_{n-1})$, we have
\[
\nu_n^0
=
\mathcal L(x_n\mid a_n=0)
=
\mathcal L(x_{n-1})
=
\mu_{n-1}.
\]
Consequently, $\nu_n^0\xrightarrow[n\to\infty]{w^*}\mu.$ Next, for a bounded and continuous $f:\mathbb R\to\mathbb R$, let
\[
\omega_f(t)
=
\sup\bigl\{
|f(x)-f(y)|:
x,y\in[0,1],\ |x-y|\leq t
\bigr\}
\]
denote its modulus of continuity on $[0,1]$. From the second identity in
\eqref{eq_convolutive_structure_non-rajchman},
\begin{align*}
\left|
\int f\,d\nu_n^1-\int f\,d\mu_{n-1}
\right|
&\leq
\frac12\omega_f\bigl(|A|\lambda^{n-2}\bigr)
+
\frac12\omega_f\bigl(|B|\lambda^{n-2}\bigr) \longrightarrow0.
\end{align*}
Since $\mu_{n-1}\xrightarrow[n\to\infty]{w^*}\mu$, it follows that
$\nu_n^1\xrightarrow[n\to\infty]{w^*}\mu.$

Now let
$\overline{\eta}
=
\frac12(\eta_0+\eta_1).$
From the definition of $\Phi_1$,
$\Phi_1
\begin{pmatrix}
\eta_0\\
\eta_1
\end{pmatrix}
=
\begin{pmatrix}
\nu_1^0*\overline{\eta}\\
\nu_1^1*\overline{\eta}
\end{pmatrix}.$
Translations commute with convolution, so induction using
\eqref{eq_convolutive_structure_non-rajchman} gives
\[
\Phi_n\cdots\Phi_1
\begin{pmatrix}
\eta_0\\
\eta_1
\end{pmatrix}
=
\begin{pmatrix}
\nu_n^0*\overline{\eta}\\
\nu_n^1*\overline{\eta}
\end{pmatrix}.
\]
The asserted operator convergence now follows from
$\nu_n^i\to\mu$ and the weak continuity of convolution with a fixed
compactly supported finite measure.

Finally, applying the operator convergence to
$(\delta_0,0)^{\mathsf T}$ and $(0,\delta_0)^{\mathsf T}$ identifies
the two columns of the weak$^*$ limit of $P_n$, yielding
\eqref{eq_matrix_Q}.
\end{proof}

We now pass to Fourier transforms. Define
\begin{equation}
\label{eq_matrix_M_n(xi)}
M_1(\xi)
=
\frac12
\begin{pmatrix}
1 & 1\\
e^{-2\pi i\xi(1-\beta)} &
e^{-2\pi i\xi(1-\beta)}
\end{pmatrix},
\end{equation}
and, for $n\geq2$, define
\[
M_n(\xi)
=
\frac12
\begin{pmatrix}
1 & 1\\
e^{-2\pi i\xi\lambda^{n-2}A} &
e^{-2\pi i\xi\lambda^{n-2}B}
\end{pmatrix}.
\]
These are the entrywise Fourier transforms of the matrices
$Q_n$.

\begin{proposition}
\label{prop_fourier_matrix_product}
For every $\xi\in\mathbb R$,
\begin{equation}
\label{eq_fourier_full_matrix_limit}
\lim_{n\to\infty}
M_n(\xi)M_{n-1}(\xi)\cdots M_1(\xi)
=
\frac{\widehat{\mu}(\xi)}{2}
\begin{pmatrix}
1 & 1\\
1 & 1
\end{pmatrix}.
\end{equation}
Moreover,
\begin{equation}
\label{eq_fourier_matrix_product}
\widehat{\mu}(\xi)
=
\lim_{n\to\infty}
\frac12(1,1)
M_n(\xi)M_{n-1}(\xi)\cdots M_2(\xi)
\begin{pmatrix}
1\\
e^{-2\pi i\xi(1-\beta)}
\end{pmatrix}.
\end{equation}
Equivalently, with the left-ordered infinite product understood in
the sense of the limit in \eqref{eq_fourier_matrix_product},
\[
\widehat{\mu}(\xi)
=
\frac12(1,1)
\left(
\overleftarrow{\prod_{k=2}^{\infty}}M_k(\xi)
\right)
\begin{pmatrix}
\widehat{\nu}_1^0(\xi)\\
\widehat{\nu}_1^1(\xi)
\end{pmatrix},
\]
where
$\widehat{\nu}_1^0(\xi)=1,
\,
\widehat{\nu}_1^1(\xi)
=
e^{-2\pi i\xi(1-\beta)}.$
\end{proposition}

\begin{proof}
Taking Fourier transforms entrywise in the weak$^*$ convergence from
Proposition \ref{prop_limit_operator} gives
\[
M_n(\xi)\cdots M_1(\xi)
=
\widehat{P_n}(\xi)
\longrightarrow
\widehat{Q}(\xi)
=
\frac{\widehat{\mu}(\xi)}{2}
\begin{pmatrix}
1 & 1\\
1 & 1
\end{pmatrix},
\]
which proves \eqref{eq_fourier_full_matrix_limit}.

On the other hand, taking Fourier transforms in
\eqref{eq_convolutive_structure_non-rajchman} gives
\[
\begin{pmatrix}
\widehat{\nu}_n^0(\xi)\\
\widehat{\nu}_n^1(\xi)
\end{pmatrix}
=
M_n(\xi)
\begin{pmatrix}
\widehat{\nu}_{n-1}^0(\xi)\\
\widehat{\nu}_{n-1}^1(\xi)
\end{pmatrix}.
\]
Iterating and using
$\widehat{\mu_n}(\xi)
=
\frac12
\left(
\widehat{\nu}_n^0(\xi)+
\widehat{\nu}_n^1(\xi)
\right)$, 
gives
\[
\widehat{\mu_n}(\xi)
=
\frac12(1,1)
M_n(\xi)\cdots M_2(\xi)
\begin{pmatrix}
1\\
e^{-2\pi i\xi(1-\beta)}
\end{pmatrix}.
\]
The result follows by letting $n\to\infty$ and using
$\mu_n\to\mu$ weakly.
\end{proof}

\subsection{Proof of Theorem
\ref{thm_main_non-conjugated}: Non-decay of the Fourier transform}

We now make a specific choice of the parameters. Pick
$\lambda=\frac1q$,
where $q\geq3$ is an integer that will be chosen sufficiently large.
In particular, $\lambda^{-1}=q$ is a Pisot number.

Choose an integer $N\geq2$ and set
\begin{equation}
\label{eq_arithmetic_choice_nonrajchman}
\varepsilon
=
\frac{\lambda^N}{1-2\lambda-\lambda^N}.
\end{equation}
For $N$ sufficiently large, the denominator is positive and
$\kappa=\lambda+\varepsilon<\frac12.$
Recall from Proposition
\ref{prop_cylinder_sizes_non-conjugated_affine} that
\[
\beta
=
\frac{\lambda^2+\kappa(1-\lambda)}
     {1-\lambda+\kappa}.
\]
Since $\kappa=\lambda+\varepsilon$, we may rewrite this as
$\beta
=
\lambda+
\frac{\varepsilon(1-2\lambda)}{1+\varepsilon}.$
Our choice \eqref{eq_arithmetic_choice_nonrajchman} therefore gives
\begin{equation}
\label{eq_beta_arithmetic_choice}
\beta=\lambda+\lambda^N.
\end{equation}
Consequently,
\begin{align}
\label{eq_AB_arithmetic_choice}
1-\beta
&=
1-\lambda-\lambda^N,
\nonumber\\
A
&=
\lambda(1-\beta)
=
\lambda-\lambda^2-\lambda^{N+1},
\\
B
&=
(1-\lambda)\beta
=
\lambda-\lambda^2+\lambda^N-\lambda^{N+1}.
\nonumber
\end{align}
In particular, $1-\beta$, $A$, and $B$ are integer polynomials in
$\lambda$.

For every $j\in\mathbb Z$, define
\begin{equation}
\label{eq_extended_matrix_nonrajchman}
\mathcal M_j
=
\frac12
\begin{pmatrix}
1&1\\
e^{-2\pi i\lambda^jA}&e^{-2\pi i\lambda^jB}
\end{pmatrix}.
\end{equation}
Also set
\begin{equation}
\label{eq_matrices_E_R_n}
E
=
\frac12
\begin{pmatrix}
1&1\\
1&1
\end{pmatrix},
\qquad
R_j:=\mathcal M_j-E.
\end{equation}
We use throughout the maximum row-sum norm on $2\times2$ matrices,
so that $\|E\|=1$.

For $t\in\mathbb R$, write
$\|t\|_{\mathbb Z}
=
\operatorname{dist}(t,\mathbb Z).$

\begin{lemma}
\label{lem_phase_summability}
There exists an absolute constant $C>0$, independent of $N$, such
that
\[
\sum_{j\in\mathbb Z}\|R_j\|
\leq
C\lambda.
\]
Moreover,
$e^{-2\pi i\lambda^{-k}(1-\beta)}=1$
for every $k\geq N$.
\end{lemma}

\begin{proof}
We use
$|e^{-2\pi it}-1|
\leq
2\pi\|t\|_{\mathbb Z}.$
It follows from \eqref{eq_matrices_E_R_n} that
\[
\|R_j\|
\leq
\pi\left(
\|\lambda^jA\|_{\mathbb Z}
+
\|\lambda^jB\|_{\mathbb Z}
\right).
\]
For $j\geq0$, the formulas in
\eqref{eq_AB_arithmetic_choice} give
$|\lambda^jA|+|\lambda^jB|
\leq
3\lambda^{j+1},$
and hence
\[
\sum_{j=0}^{\infty}
\left(
\|\lambda^jA\|_{\mathbb Z}
+
\|\lambda^jB\|_{\mathbb Z}
\right)
\leq
\frac{3\lambda}{1-\lambda}.
\]
Now write $j=-r$, where $r\geq1$. By
\eqref{eq_AB_arithmetic_choice},
$\lambda^{-r}A
=
q^{r-1}-q^{r-2}-\lambda^{N+1-r}.$
Since $N\geq2$ and $\lambda\leq1/3$, 
\[
\|\lambda^{-r}A\|_{\mathbb Z}
=
\begin{cases}
\lambda+\lambda^N,
& r=1,\\[1mm]
\lambda^{N+1-r},
& 2\leq r\leq N,\\[1mm]
0,
& r\geq N+1.
\end{cases}
\]
Indeed, when $2\leq r\leq N$, the first two terms are integers and
$0<\lambda^{N+1-r}\leq\lambda<1/2$, while for $r\geq N+1$ all
three terms are integers. Therefore,
\[
\begin{aligned}
\sum_{r=1}^{\infty}
\|\lambda^{-r}A\|_{\mathbb Z}
&=
\lambda+\lambda^N+
\sum_{r=2}^{N}\lambda^{N+1-r}\\
&=
\lambda+\lambda^N+
\sum_{s=1}^{N-1}\lambda^s\\
&\leq
2\lambda+\frac{\lambda^2}{1-\lambda}
\leq
3\lambda.
\end{aligned}
\]
Similarly,
$\lambda^{-r}B
=
q^{r-1}-q^{r-2}
+
(1-\lambda)\lambda^{N-r}.$
For $r=1$, we have
\[
\lambda^{-1}B
=
1-\left(\lambda-(1-\lambda)\lambda^{N-1}\right),
\]
where
$0\leq
\lambda-(1-\lambda)\lambda^{N-1}
\leq\lambda<\frac12.$
So,
\[
\|\lambda^{-r}B\|_{\mathbb Z}
=
\begin{cases}
\lambda-(1-\lambda)\lambda^{N-1},
& r=1,\\[1mm]
(1-\lambda)\lambda^{N-r},
& 2\leq r\leq N-1,\\[1mm]
\lambda,
& r=N,\\[1mm]
0,
& r\geq N+1.
\end{cases}
\]
Here, the middle range is understood to be empty when $N=2$. Hence,
\[
\begin{aligned}
\sum_{r=1}^{\infty}
\|\lambda^{-r}B\|_{\mathbb Z}
&\leq
\lambda+
(1-\lambda)\sum_{s=1}^{N-2}\lambda^s
+\lambda\\
&\leq
3\lambda.
\end{aligned}
\]

Combining these estimates with the estimate for $j\geq0$, we obtain
\[
\begin{aligned}
\sum_{j\in\mathbb Z}
\left(
\|\lambda^jA\|_{\mathbb Z}
+
\|\lambda^jB\|_{\mathbb Z}
\right)
&\leq
6\lambda+\frac{3\lambda}{1-\lambda}\\
&\leq
12\lambda,
\end{aligned}
\]
since $\lambda<1/2$. Thus one may take $C=12\pi$. 

Finally, by \eqref{eq_beta_arithmetic_choice},
\[
\lambda^{-k}(1-\beta)
=
q^k-q^{k-1}-q^{k-N},
\]
which is an integer whenever $k\geq N$. Therefore,
$e^{-2\pi i\lambda^{-k}(1-\beta)}=1
\text{ for every }k\geq N.$
\end{proof}

We next record a standard elementary fact about products of summable
perturbations of an idempotent matrix.

\begin{lemma}
\label{lem_limit_left-handed_product}
\label{lem_limit_right-handed_product}
The limits
\begin{equation}
\label{eq_finite_1st_coefficient}
L
:=
\lim_{m\to\infty}
\mathcal M_m\mathcal M_{m-1}\cdots\mathcal M_1,
\quad
Q
:=
\lim_{k\to\infty}
\mathcal M_0\mathcal M_{-1}\cdots\mathcal M_{-k}
\end{equation}
exist in the operator norm. Moreover, the ordered two-sided product
$G:=LQ$
satisfies
\begin{equation}
\label{eq_bilateral_product_close_to_E}
\|G-E\|
\leq
e^S-1,
\text{ where }\quad 
S:=\sum_{j\in\mathbb Z}\|R_j\|.
\end{equation}
\end{lemma}

\begin{proof}
Recall that
$\mathcal M_j=E+R_j,
\,
E^2=E,
\,
\|E\|=1.$
Set
\[
L_m:=\mathcal M_m\mathcal M_{m-1}\cdots\mathcal M_1.
\]
Since
$\|\mathcal M_j\|
\leq
1+\|R_j\|,$
we have the uniform estimate
\[
\|L_m\|
\leq
\prod_{j=1}^m(1+\|R_j\|)
\leq
\exp\left(\sum_{j=1}^{\infty}\|R_j\|\right)
\leq e^S.
\]
Moreover,
\[
\begin{aligned}
EL_{m+1}-EL_m
&=
E(E+R_{m+1})L_m-EL_m\\
&=
ER_{m+1}L_m.
\end{aligned}
\]
So,
$\|EL_{m+1}-EL_m\|
\leq
e^S\|R_{m+1}\|.$
Since
$\sum_{m=1}^{\infty}\|R_m\|<\infty,$
the sequence $(EL_m)_{m\geq1}$ is Cauchy and therefore converges in
the operator norm. On the other hand,
\[
\begin{aligned}
L_m-EL_m
&=
(E+R_m)L_{m-1}
-
E(E+R_m)L_{m-1}\\
&=
(I-E)R_mL_{m-1}.
\end{aligned}
\]
Thus,
$\|L_m-EL_m\|
\leq
\|I-E\|e^S\|R_m\|
\longrightarrow0.$
It follows that $(L_m)_{m\geq1}$ converges in the operator norm. We
denote its limit by $L$.

Similarly, set
$Q_k:=\mathcal M_0\mathcal M_{-1}\cdots\mathcal M_{-k}.$
The preceding argument, with multiplication by $E$ on the right
instead of on the left, applies verbatim. Indeed,
\[
\|Q_k\|\leq e^S,\quad
Q_{k+1}E-Q_kE
=
Q_kR_{-(k+1)}E,
\quad
Q_k-Q_kE
=
Q_{k-1}R_{-k}(I-E).
\]
Hence,
$\|Q_{k+1}E-Q_kE\|
\leq
e^S\|R_{-(k+1)}\|,$
so $(Q_kE)_{k\geq0}$ is Cauchy, while
$\|Q_k-Q_kE\|
\leq
e^S\|I-E\|\|R_{-k}\|
\rightarrow0$
Therefore $(Q_k)_{k\geq0}$ converges in the operator norm; denote its
limit by $Q$.

To prove \eqref{eq_bilateral_product_close_to_E}, for
$m,k\geq0$, let
$P_{m,k}
:=
\mathcal M_m\cdots\mathcal M_1
\mathcal M_0\mathcal M_{-1}\cdots\mathcal M_{-k}.$
Expanding
$$P_{m,k}
=
(E+R_m)\cdots(E+R_{-k}),$$
each term is determined by a subset
$F\subseteq\{-k,\ldots,m\}$: at the indices in $F$ we choose $R_j$,
and at all remaining indices we choose $E$, with the factors kept in
their original order. Denote the resulting term by $T_F$. Then
$T_\varnothing=E$, while, for $F\neq\varnothing$,
$\|T_F\|
\leq
\prod_{j\in F}\|R_j\|,$
since $\|E\|=1$. Consequently,
\[
\begin{aligned}
\|P_{m,k}-E\|
&\leq
\sum_{\substack{F\subseteq\{-k,\ldots,m\}\\F\neq\varnothing}}
\prod_{j\in F}\|R_j\|\\
&=
\prod_{j=-k}^{m}(1+\|R_j\|)-1\\
&\leq
\exp\left(\sum_{j=-k}^{m}\|R_j\|\right)-1\\
&\leq
e^S-1
\end{aligned}
\]
Since
$P_{m,k}=L_mQ_k$
and $L_m\to L$, $Q_k\to Q$ in the operator norm, we have
$P_{m,k}\longrightarrow LQ=G$
as $m,k\to\infty$. Passing to the limit in the preceding inequality
gives
$\|G-E\|\leq e^S-1,$
as required.
\end{proof}
We now choose the integer $q$ sufficiently large that
\begin{equation}
\label{eq_small_bilateral_perturbation}
e^{C\lambda}-1<\frac12,
\end{equation}
where $C$ is the constant from Lemma
\ref{lem_phase_summability}. Fix such a $q$, and then choose $N$
sufficiently large that the parameter $\varepsilon$ in
\eqref{eq_arithmetic_choice_nonrajchman} satisfies
$\kappa<1/2$.

We relate the products above to the Fourier transform of $\mu$. Recall
that, for $n\geq2$,
\[
M_n(\xi)
=
\frac12
\begin{pmatrix}
1&1\\
e^{-2\pi i\xi\lambda^{n-2}A}
&
e^{-2\pi i\xi\lambda^{n-2}B}
\end{pmatrix}.
\]
Hence, for every $k\geq1$,
$M_n(\lambda^{-k})
=
\mathcal M_{n-k-2}.$
Consequently, for every $m\geq2$,
\begin{equation}
\label{eq_Erdos_like_factorization}
M_m(\lambda^{-k})
M_{m-1}(\lambda^{-k})
\cdots
M_2(\lambda^{-k})
=
\mathcal M_{m-k-2}
\mathcal M_{m-k-3}
\cdots
\mathcal M_{-k}.
\end{equation}
Letting $m\to\infty$ and using
Lemma \ref{lem_limit_left-handed_product}, we obtain
\[
\lim_{m\to\infty}
M_m(\lambda^{-k})\cdots M_2(\lambda^{-k})
=
LQ_k,
\]
where
$Q_k
:=
\mathcal M_0\mathcal M_{-1}\cdots\mathcal M_{-k}.$
Moreover,
\[
Q_k\longrightarrow Q
\qquad\text{and hence}\qquad
LQ_k\longrightarrow G=LQ.
\]

We can now complete the proof of the Theorem.

\begin{proof}[Proof of Theorem
\ref{thm_main_non-conjugated}]
For every fixed $k\geq N$, Proposition
\ref{prop_fourier_matrix_product} and Lemma
\ref{lem_phase_summability} give
\[
\widehat{\mu}(\lambda^{-k})
=
\lim_{m\to\infty}
\frac12(1,1)
M_m(\lambda^{-k})\cdots M_2(\lambda^{-k})
\mathbf e,
\qquad
\mathbf e:=
\begin{pmatrix}
1\\
1
\end{pmatrix}.
\]
By \eqref{eq_Erdos_like_factorization}, for all sufficiently large
$m$,
\[
M_m(\lambda^{-k})\cdots M_2(\lambda^{-k})
=
\bigl(
\mathcal M_{m-k-2}\cdots\mathcal M_1
\bigr)
Q_k.
\]
Since
$\mathcal M_{m-k-2}\cdots\mathcal M_1
\rightarrow L
\text{ as }m\to\infty,$
we obtain
$\widehat{\mu}(\lambda^{-k})
=
\frac12(1,1)LQ_k\mathbf e.$
Finally, $Q_k\to Q$ in operator norm. Hence
\begin{equation}
\label{eq_Frourier_coeff_convergence}
\widehat{\mu}(\lambda^{-k})
\longrightarrow
\rho
:=
\frac12(1,1)LQ\mathbf e
=
\frac12(1,1)G\mathbf e,
\qquad k\to\infty.
\end{equation}
Since $E\mathbf e=\mathbf e$, we have
$\frac12(1,1)E\mathbf e=1.$ Since
we are using the maximum row-sum norm,
\[
\left|
\rho-1
\right|
=
\frac12
\left|
(1,1)(G-E)\mathbf e
\right|
\leq
\|G-E\|.
\]
By Lemmas \ref{lem_phase_summability} and
\ref{lem_limit_left-handed_product}, together with
\eqref{eq_small_bilateral_perturbation},
\[
\|G-E\|
\leq
e^S-1
\leq
e^{C\lambda}-1
<
\frac12.
\]
It follows that
\[
|\rho|
\geq
1-|\rho-1|
>
\frac12.
\]
Therefore
\[
\widehat{\mu}(\lambda^{-k})
\longrightarrow\rho\neq0
\qquad\text{as }k\to\infty.
\]
Since $\lambda^{-k}\to\infty$, the measure $\mu$ is not Rajchman.

Proposition \ref{prop_construction_IFS_sec_3} shows that the
underlying IFS is $C^\infty$, pseudo-affine, and not
$C^1$-conjugate to a self-similar IFS. This completes the proof of
Theorem \ref{thm_main_non-conjugated}.
\end{proof}
%---------------------------------------------------------------------

\section{Proof of Theorem \ref{thm_main-conjugated_poly-decay}}
\label{section_proof_conjugated}

Fix $0<\lambda<1/2$, an integer $r\geq1$, and
$0<\alpha\leq1$, and write
$s:=r+\alpha.$
The proof has three steps. First, we construct a random homogeneous
Cantor set. Here, homogeneous means that all cylinders of the same
level have the same diameter. At the $n$-th stage, we prescribe the
common displacement between the left endpoints of the two children of
each level-$n$ cylinder. Before normalization, these displacements are
given by
\[
d_n(\omega)
=
(1-\lambda)\lambda^n
\left(
1+a_nY_n(\omega)
\right),
\qquad
a_n=c\lambda^{(s-1)n},
\qquad n\geq0,
\]
where the $Y_n$ are independent and uniformly distributed on
$[-1,1]$. We then normalize by
\[
L_0(\omega):=\sum_{n=0}^{\infty}d_n(\omega),
\qquad
\widetilde d_n(\omega):=\frac{d_n(\omega)}{L_0(\omega)},
\]
so that $\sum_{n\geq0}\widetilde d_n(\omega)=1$. Thus the geometry is
a level-dependent random perturbation of the self-similar set
$X_{\Phi_\lambda}$, whose relative size at level $n$ is
$O(\lambda^{(s-1)n})$. We show that the resulting dynamical
proportions satisfy
\[
\lambda_i(w)
=
\lambda+O\left(\lambda^{(s-1)|w|}\right).
\]
The results from \cite{ABRHS_pseudo-Affine_IFS}, recalled in Section
\ref{subsection_pseudo-aff_Cantor_sets}, then produce a
$C^{r,\alpha}$ pseudo-affine IFS $\Phi_\omega$ with slope $\lambda$
and a $C^{r,\alpha}$ conjugacy $h_\omega$ between
$\Phi_\lambda$ and $\Phi_\omega$. Moreover, $h_\omega'$ is constant
on $X_{\Phi_\lambda}$.

Second, the uniform stationary measure $\nu_\omega$ of
$\Phi_\omega$ is the law of
\[
\sum_{n=0}^{\infty}
\epsilon_n\widetilde d_n(\omega),
\]
where the $\epsilon_n$ are independent and uniformly distributed on
$\{0,1\}$. Consequently,
\[
\left|\widehat{\nu_\omega}(\xi)\right|
=
\prod_{n=0}^{\infty}
\left|
\cos\left(\pi\xi\widetilde d_n(\omega)\right)
\right|.
\]
For the Fourier-decay argument, it is convenient to remove the random
normalization and work instead with the law $\eta_\omega$ of
\[
\sum_{n=0}^{\infty}\epsilon_nd_n(\omega).
\]
The two measures differ only by a dilation, since
\[
\widehat{\nu_\omega}(\xi)
=
\widehat{\eta_\omega}
\left(\frac{\xi}{L_0(\omega)}\right).
\]

Finally, whenever $|\xi|\lambda^{sn}$ is bounded below, the random
variable $Y_n$ moves the phase $\xi d_n(\omega)$ through an interval
of controlled length. A one-scale moment estimate therefore gives a
uniform contraction for the expected size of the corresponding
Fourier factor. Multiplying these estimates over the relevant scales
yields polynomial decay in expectation. Markov's inequality and a
Borel--Cantelli argument on suitable frequency nets, followed by the
uniform Lipschitz bound for the Fourier transform, give
\[
\left|\widehat{\nu_\omega}(\xi)\right|
\lesssim_\omega
|\xi|^{-\tau}
\]
for almost every $\omega$, where $\tau$ is the exponent stated in
Theorem \ref{thm_main-conjugated_poly-decay}. Choosing one such
realization completes the proof.
\subsection{The Cantor set construction, and Riesz product of the stationary measure} \label{Section: Canot set construction}
We begin by constructing a random Cantor set on $[0,1]$.

Let
$\Omega=[-1,1]^{\mathbb N_0},
\,
\mathbb P=
\left(\frac12\operatorname{Leb}|_{[-1,1]}\right)^{\mathbb N_0},$
and let $Y_n:\Omega\to[-1,1]$ be the coordinate functions.  Choose
$c>0$ sufficiently small, depending only on $\lambda$, and set
\begin{equation}
\label{eq_u_n-(r,alpha)-case}
a_n:=c\lambda^{(s-1)n},
\qquad
d_n(\omega):=(1-\lambda)\lambda^n
\bigl(1+a_nY_n(\omega)\bigr),
\qquad n\geq0.
\end{equation}
We assume, in particular, that
\begin{equation}
\label{eq_choice_c_conjugated}
0<c<\frac12,
\qquad
\lambda\frac{1+c}{1-c}<\frac12.
\end{equation}
We may take $c$ even smaller as the construction progresses. 

For $n\geq0$, define
$L_n(\omega):=\sum_{m=n}^{\infty}d_m(\omega).$
Then
\begin{equation}
\label{eq_Ln_conjugated}
L_n(\omega)=\lambda^n\bigl(1+v_n(\omega)\bigr),
\qquad
v_n(\omega):=(1-\lambda)
\sum_{j=0}^{\infty}\lambda^j a_{n+j}Y_{n+j}(\omega).
\end{equation}
Since $(a_n)$ is decreasing,
\begin{equation}
\label{eq_un_bound_conjugated}
|v_n(\omega)|
\leq
(1-\lambda)\sum_{j=0}^{\infty}\lambda^j a_{n+j}
\leq a_n.
\end{equation}
In particular, $L_0(\omega)>0$.  Normalize by setting
\begin{equation}
\label{eq_normalized_lengths_conjugated}
\ell_n(\omega):=\frac{L_n(\omega)}{L_0(\omega)},
\qquad
\widetilde d_n(\omega):=\frac{d_n(\omega)}{L_0(\omega)}.
\end{equation}
Then
\[
\ell_0=1,
\qquad
\widetilde d_n=\ell_n-\ell_{n+1},
\qquad
\sum_{n=0}^{\infty}\widetilde d_n=1.
\]

For each $\omega\in\Omega$, we now define cylinders
$K_w(\omega)\subset[0,1]$, indexed by $w\in\W:=\{0,1\}^*$.  Set
$K_e(\omega)=[0,1]$.  If $|w|=n$ and
$K_w(\omega)=[z_w,z_w+\ell_n],$
define
\begin{align}
\label{eq_recursive_cylinders_conjugated}
K_{w0}(\omega)
&=[z_w,z_w+\ell_{n+1}],
\nonumber\\
K_{w1}(\omega)
&=[z_w+\widetilde d_n,z_w+\widetilde d_n+\ell_{n+1}]
=[z_w+\ell_n-\ell_{n+1},z_w+\ell_n].
\end{align}
Let
\begin{equation}
\label{eq_Xomega_construction}
X_\omega
:=
\bigcap_{n=0}^{\infty}
\bigcup_{|w|=n}K_w(\omega).
\end{equation}
Thus all cylinders of level $n$ have diameter $\ell_n$.  If $|w|=n$,
the gap between $K_{w0}$ and $K_{w1}$ is
\[
I_w(\omega)
=
(z_w+\ell_{n+1},z_w+\widetilde d_n),
\]
and has length
\begin{equation}
\label{eq_gap_length_conjugated}
|I_w(\omega)|
=
\ell_n-2\ell_{n+1}
=
(1-2q_n)\ell_n,
\qquad
q_n:=\frac{\ell_{n+1}}{\ell_n}.
\end{equation}
By \eqref{eq_Ln_conjugated}--\eqref{eq_normalized_lengths_conjugated},
\begin{equation}
\label{eq_qn_conjugated}
q_n
=
\lambda\frac{1+v_{n+1}}{1+v_n}.
\end{equation}
The choice \eqref{eq_choice_c_conjugated} and
\eqref{eq_un_bound_conjugated} imply that
\[
0<q_n<\frac12
\]
uniformly in $n$ and $\omega$, so all the gaps above are nonempty.

\begin{proposition}
\label{prop_construction_IFS_conjugated_to_affine_finite_regularity}
For every $\omega\in\Omega$, there exist a hyperbolic $\lambda$-pseudo-affine IFS
$\Phi_\omega=\{f_{0,\omega},f_{1,\omega}\}$
of class $C^{r,\alpha}$, and a $C^{r,\alpha}$
diffeomorphism $h_\omega:[0,1]\to[0,1]$, with the following
properties:
\begin{enumerate}
    \item the attractor of $\Phi_\omega$ is $X_\omega$;

    \item $h_\omega$ conjugates $\Phi_\lambda$ to $\Phi_\omega$ on
    their attractors, and
    \begin{equation}
    \label{eq_constant_derivative_conjugated}
    h_\omega'(x)=\frac1{L_0(\omega)}
    \qquad
    \text{for every }x\in X_{\Phi_\lambda};
    \end{equation}

    \item The uniform stationary measure $\nu_\omega$ of
    $\Phi_\omega$ satisfies
    $\nu_\omega=(h_\omega)_*\mu_\lambda;$

    \item the Fourier transform of $\nu_\omega$ satisfies
    \begin{equation}
    \label{eq_riesz_product_completely_symmetric}
    \widehat{\nu_\omega}(\xi)
    =
    \prod_{n=0}^{\infty}
    \frac{1+e^{-2\pi i\xi\widetilde d_n(\omega)}}2
    =
    e^{-\pi i\xi}
    \prod_{n=0}^{\infty}
    \cos\bigl(\pi\xi\widetilde d_n(\omega)\bigr).
    \end{equation}
\end{enumerate}
\end{proposition}

\begin{proof}
For $|w|=n$, the dynamical proportions determined by the gaps in
\eqref{eq_gap_length_conjugated} are independent of both $w$ and the
choice of $i\in\{0,1\}$. We denote their common value by 
\begin{equation} \label{eq_rn_conjugated} 
r_n := \lambda_i(w) = \frac{|I_{iw}|}{|I_w|} = q_n\frac{1-2q_{n+1}}{1-2q_n}. 
\end{equation}
Write
$r_n=\lambda+\theta_n,
\,
\theta_i(w):=\theta_{|w|}
\quad(i=0,1).$
By \eqref{eq_qn_conjugated} and
\eqref{eq_un_bound_conjugated},
\[
\begin{aligned}
|q_n-\lambda|
&=
\lambda
\left|
\frac{v_{n+1}-v_n}{1+v_n}
\right|\\
&\leq
\frac{\lambda}{1-c}
\bigl(|v_n|+|v_{n+1}|\bigr)
\leq
\frac{2\lambda}{1-c}a_n
\leq
2a_n.
\end{aligned}
\]
Since the numbers $q_n$ remain in a fixed compact subinterval of
$(0,1/2)$, \eqref{eq_rn_conjugated} gives
\begin{equation}
\label{eq_theta_decay_conjugated}
|\theta_n|
\leq
C_\lambda\bigl(|q_n-\lambda|+|q_{n+1}-\lambda|\bigr)
\leq C_\lambda\lambda^{(s-1)n}, \text{ for some } C_\lambda>0.
\end{equation}
After decreasing $c$  more, if necessary, we also have
$0<r_n<1$ for every $n$ and $\omega$.

Moreover, by \eqref{eq_Ln_conjugated},
\eqref{eq_normalized_lengths_conjugated}, and
\eqref{eq_gap_length_conjugated},
\begin{equation}
\label{eq_gap_comparable_lambda_n}
|I_w|
\asymp_{\lambda,c}
\lambda^n
\text{ for }
|w|=n, \text{ uniformly in } \omega.
\end{equation}
If $\Psi_\omega$ denotes the cocycle associated
to the proportions in \eqref{eq_rn_conjugated}, then
$\Psi_\omega(w)=\frac{|I_w|}{|I_e|}
\asymp_{\lambda,c}\lambda^{|w|}.$
Thus, \eqref{eq_theta_decay_conjugated} implies
$|\theta_{|w|}|
\leq
C_\lambda\Psi_\omega(w)^{s-1}.$
Theorem \ref{thm_classification_pAffine_IFS} then produces a
hyperbolic $\lambda$-pseudo-affine IFS $\Phi_\omega$ of class $C^{r,\alpha}$,
with the prescribed dynamical proportions.
The cylinder and gap lengths of its attractor are determined by these
proportions through Proposition
\ref{prop_proportion_functions_and_Cantor_set}. Since the left child
of each cylinder shares its left endpoint and the right child shares
its right endpoint, induction on the level shows that its cylinder
intervals coincide with the intervals $K_w(\omega)$ constructed in
\eqref{eq_recursive_cylinders_conjugated}. Hence 
$X_{\Phi_\omega} =X_\omega$.

We next identify the coding map.  If
$a=(a_n)_{n\geq1}\in\Sigma_2^+$,
then repeated use of \eqref{eq_recursive_cylinders_conjugated} shows
that the left endpoint of $K_{a_1\cdots a_N}$ is
$\sum_{n=0}^{N-1}a_{n+1}\widetilde d_n.$
Since $\ell_N\to0$, it follows that
\begin{equation}
\label{eq_coding_sum_conjugated}
\Theta_{\Phi_\omega}(a)
=
\sum_{n=0}^{\infty}a_{n+1}\widetilde d_n(\omega).
\end{equation}
Therefore the pushforward of the uniform Bernoulli measure under
$\Theta_{\Phi_\omega}$ is the law of
$\sum_{n=0}^{\infty}\epsilon_n\widetilde d_n(\omega),$
where the $\epsilon_n$ are independent and uniform on $\{0,1\}$.
This proves the product formula in
\eqref{eq_riesz_product_completely_symmetric}.

It remains to prove the conjugacy statement.  For every word $w$ of
length $n$, the cocycle telescopes as
\[
\Psi_\omega(w)
=
\prod_{k=0}^{n-1}r_k
=
\frac{|I_w|}{|I_e|}.
\]
Using \eqref{eq_gap_length_conjugated}, we obtain
\begin{equation}
\label{eq_cocycle_ratio_conjugated}
\frac{\Psi_\omega(w)}{\lambda^n}
=
\frac{1-2q_n}{1-2q_0}
\frac{\ell_n}{\lambda^n}
=
\frac{1-2q_n}{1-2q_0}
\frac{1+v_n}{L_0}.
\end{equation}
Since $v_n\to0$ and $q_n\to\lambda$, the right-hand side converges
to
\begin{equation}
\label{eq_chi_limit_conjugated}
\chi_\omega
=
\frac{1-2\lambda}{(1-2q_0)L_0},
\end{equation}
independently of the word and hence independently of the symbolic
address.  Theorem \ref{them_criterion_for_conjgation} now gives a
$C^{r,\alpha}$ extension of the address-preserving map
$h_\omega
=
\Theta_{\Phi_\omega}
\circ
\Theta_{\Phi_\lambda}^{-1}.$
Since the central gaps of $X_\omega$ and $X_{\Phi_\lambda}$ have
lengths $1-2q_0$ and $1-2\lambda$, respectively, the derivative formula
in that theorem and \eqref{eq_chi_limit_conjugated} yield
\[
h_\omega'
=
\frac{1-2q_0}{1-2\lambda}\chi_\omega
=
\frac1{L_0}
\]
on $X_{\Phi_\lambda}$.  This proves
\eqref{eq_constant_derivative_conjugated}.  Finally, because
$h_\omega$ preserves symbolic addresses, pushing forward the uniform
Bernoulli measure gives
$\nu_\omega=(h_\omega)_*\mu_\lambda.$
\end{proof}

\subsection{Polynomial Fourier decay} \label{Section: poly decay for cong}

We now prove that the random measures $\nu_\omega$ from Proposition
\ref{prop_construction_IFS_conjugated_to_affine_finite_regularity}
have polynomial Fourier decay almost surely, with the exponent $\tau$ from
 Theorem \ref{thm_main-conjugated_poly-decay}.

For the proof, it is convenient to remove the random normalization.
Let $\eta_\omega$ be the law of
$\sum_{n=0}^{\infty}\epsilon_n d_n(\omega),$
where the $\epsilon_n$ are independent and uniform on $\{0,1\}$.  Then
\begin{equation}
\label{eq_eta_model_product}
|\widehat{\eta_\omega}(\xi)|
=
\prod_{n=0}^{\infty}
\left|
\cos\bigl(\pi\xi d_n(\omega)\bigr)
\right|,
\end{equation}
and, by \eqref{eq_normalized_lengths_conjugated},
\begin{equation}
\label{eq_scaling_eta_nu}
\widehat{\nu_\omega}(\xi)
=
\widehat{\eta_\omega}\left(\frac{\xi}{L_0(\omega)}\right).
\end{equation}
Since $1-c\leq L_0(\omega)\leq1+c$, it is enough to prove a uniform
power bound for $\widehat{\eta_\omega}$.

We begin with a one-scale moment estimate.

\begin{lemma}
\label{lem_one_scale_moment_conjugated}
Let $Z$ be uniformly distributed on $[-1,1]$.  If $p\geq1$,
$0<a\leq1/2$, and $u>0$ satisfy $au\geq1$, then
\begin{equation}
\label{eq_one_scale_moment_conjugated}
\mathbb E
\left|
\cos\bigl(\pi u(1+aZ)\bigr)
\right|^p
\leq
\rho_p,
\qquad
\rho_p:=\sqrt{\frac{9}{2\pi p}}.
\end{equation}
\end{lemma}

\begin{proof}
As $Z$ ranges over $[-1,1]$, the variable $t=u(1+aZ)$ ranges over an
interval $J$ of length $2au$.  Hence
\[
\mathbb E
\left|
\cos\bigl(\pi u(1+aZ)\bigr)
\right|^p
=
\frac1{2au}
\int_J|\cos(\pi t)|^p\,dt.
\]
Since $t\mapsto|\cos(\pi t)|^p$ is $1$-periodic and nonnegative,
\[
\int_J|\cos(\pi t)|^p\,dt
\leq
(2au+1)
\int_0^1|\cos(\pi t)|^p\,dt.
\]
The beta integral gives
\[
\int_0^1|\cos(\pi t)|^p\,dt
=
\frac1{\sqrt\pi}
\frac{\Gamma((p+1)/2)}{\Gamma((p+2)/2)}
\leq
\sqrt{\frac{2}{\pi p}},
\]
where the last inequality follows from the log-convexity of the gamma
function.
Since $au\geq1$, the preceding three displays imply
\[
\mathbb E
\left|
\cos\bigl(\pi u(1+aZ)\bigr)
\right|^p
\leq
\frac32\sqrt{\frac{2}{\pi p}}
=
\sqrt{\frac{9}{2\pi p}}.
\]
\end{proof}

Set
$\beta:=\lambda^{-1}>1.$
For $N\geq1$, define
\begin{equation}
\label{eq_good_scales_moment_conjugated}
G_N
:=
\left\{
0\leq m\leq N:
 a_{N-m}\beta^m\geq1
\right\}.
\end{equation}
Since
$a_{N-m}\beta^m
=
c\beta^{sm-(s-1)N},$
there exists a constant $C_0=C_0(\lambda,s,c)$ such that
\begin{equation}
\label{eq_number_good_scales_conjugated}
|G_N|
\geq
\frac Ns-C_0
\qquad
\text{for every }N\geq1.
\end{equation}

Fix $\xi>0$, and choose $N\geq0$ and $y\in[1,\beta]$ so that
\begin{equation}
\label{eq_frequency_decomposition_conjugated}
(1-\lambda)\xi=\beta^Ny.
\end{equation}
For $m\in G_N$, put $n=N-m$.  Then
\[
\xi d_n
=
y\beta^m\bigl(1+a_{N-m}Y_{N-m}\bigr), \text{ and  }
a_{N-m}y\beta^m\geq1.
\]
For distinct $m\in G_N$, the factors
$\left|
\cos\left(
\pi y\beta^m
\bigl(1+a_{N-m}Y_{N-m}\bigr)
\right)
\right|^p$
depend on the distinct random variables $Y_{N-m}$. Since the family
$(Y_n)_{n\geq0}$ is independent, these factors are independent. 
So, Lemma \ref{lem_one_scale_moment_conjugated} and
\eqref{eq_eta_model_product} give
\begin{equation}
\label{eq_fixed_frequency_moment_conjugated}
\mathbb E
|\widehat{\eta_\omega}(\xi)|^p
\leq
\rho_p^{|G_N|}
\leq
C_p\rho_p^{N/s}, \text{ where    } C_p=\rho_p^{-C_0}.
\end{equation}

\begin{proposition}
\label{prop_random_statement_finite_regularity}
For $\mathbb P$-almost every $\omega$, there is a constant
$C_\omega>0$ such that
\begin{equation}
\label{eq_final_fourier_decay_conjugated}
|\widehat{\nu_\omega}(\xi)|
\leq
C_\omega|\xi|^{-\tau},
\quad
|\xi|\geq1,
\end{equation}
where
\begin{equation}
\label{eq_tau_conjugated}
\tau
=
\frac{2\pi\lambda^{2s}}
{(18e+\pi)s\log(1/\lambda)}.
\end{equation}
\end{proposition}

\begin{proof}
Choose
\begin{equation}
\label{eq_choice_p_conjugated}
p
:=
\frac{9e}{2\pi}\lambda^{-2s}
=
\frac{9e}{2\pi}\beta^{2s}.
\end{equation}
Then Lemma \ref{lem_one_scale_moment_conjugated} gives
\begin{equation}
\label{eq_rhop_choice_conjugated}
\rho_p
\leq
\sqrt{\frac{9}{2\pi p}}
=
e^{-1/2}\beta^{-s}.
\end{equation}

For $N\geq1$, let
$\mathcal A_N
:=
\left[
\frac{\beta^N}{1-\lambda},
\frac{\beta^{N+1}}{1-\lambda}
\right].$
The measure $\eta_\omega$ is supported on $[0,L_0(\omega)]$, and
$L_0(\omega)\leq1+c<3/2$.  Therefore, differentiation under the integral gives
\begin{equation}
\label{eq_lipschitz_eta_conjugated}
|\widehat{\eta_\omega}'(\xi)|
\leq
2\pi L_0(\omega)
\leq3\pi.
\end{equation}
Choose a finite mesh $\mathcal Z_N\subset\mathcal A_N$ with spacing at
most
$h_N:=\frac1{6\pi}\beta^{-\tau N}.$
Then there exists a uniform $C>0$ with
\begin{equation}
\label{eq_mesh_cardinality_conjugated}
|\mathcal Z_N|
\leq
C\beta^{(1+\tau)N}.
\end{equation}
Let $F_N$ be the event that there exists $\zeta\in\mathcal Z_N$ for
which
$|\widehat{\eta_\omega}(\zeta)|
>
\frac12\beta^{-\tau N}.$
By Markov's inequality, \eqref{eq_fixed_frequency_moment_conjugated},
 \eqref{eq_mesh_cardinality_conjugated}, using also
\eqref{eq_number_good_scales_conjugated}, and recalling that
$\rho_p<1$, we obtain
\[
\begin{aligned} 
\mathbb P(F_N)
&\leq
C\beta^{(1+\tau)N}
2^p\beta^{p\tau N}
\rho_p^{N/s-C_0}\\
&=
C'
\exp\left(
N\left(
[1+(p+1)\tau]\log\beta
+\frac1s\log\rho_p
\right)
\right), 
\end{aligned}
\]
where
$C':=C\,2^p\rho_p^{-C_0}.$
By \eqref{eq_rhop_choice_conjugated},
$\frac1s\log\rho_p
\leq
-\log\beta-\frac1{2s}.$
Hence the exponent  is at most
\[
N\left(
(p+1)\tau\log\beta-\frac1{2s}
\right).
\]
For the value of $\tau$ in \eqref{eq_tau_conjugated},
\begin{align*}
(p+1)\tau\log\beta
&=
\left(
\frac{9e}{2\pi}\lambda^{-2s}+1
\right)
\frac{2\pi\lambda^{2s}}
{(18e+\pi)s}
\\
&=
\frac{9e+2\pi\lambda^{2s}}
{(18e+\pi)s}
<
\frac1{2s},
\end{align*}
because $\lambda^{2s}<1/4$.  Therefore
\[
\sum_{N=1}^{\infty}\mathbb P(F_N)<\infty.
\]
By the Borel--Cantelli lemma, for almost every $\omega$ there exists
$N_0(\omega)$ such that $F_N$ does not occur whenever
$N\geq N_0(\omega)$.

Fix such an $\omega$ and $N\geq N_0(\omega)$.  Given
$\xi\in\mathcal A_N$, choose $\zeta\in\mathcal Z_N$ with
$|\xi-\zeta|\leq h_N$.  By
\eqref{eq_lipschitz_eta_conjugated},
\[
|\widehat{\eta_\omega}(\xi)|
\leq
|\widehat{\eta_\omega}(\zeta)|
+3\pi h_N
\leq
\beta^{-\tau N}.
\]
Since $\xi\in\mathcal A_N$, this implies
\[
|\widehat{\eta_\omega}(\xi)|
\leq
C_{\lambda,\tau}|\xi|^{-\tau}.
\]
The finitely many shells with $N<N_0(\omega)$ are absorbed into a
random constant.  The estimate for negative frequencies follows by
complex conjugation.  Finally, \eqref{eq_scaling_eta_nu} and the
uniform bounds on $L_0(\omega)$ give
\eqref{eq_final_fourier_decay_conjugated}.
\end{proof}

\begin{proof}[Proof of Theorem
\ref{thm_main-conjugated_poly-decay}]
Choose $\omega$ in the full-probability set supplied by Proposition
\ref{prop_random_statement_finite_regularity}, and set
\[
g:=h_\omega.
\]
By Proposition
\ref{prop_construction_IFS_conjugated_to_affine_finite_regularity},
the measure $g_*\mu_\lambda=\nu_\omega$ is self-conformal and
$g'$ is constant on $\operatorname{supp}(\mu_\lambda)$.  Proposition
\ref{prop_random_statement_finite_regularity} gives the Fourier-decay
estimate with the exponent in \eqref{eq_tau_conjugated}.  This proves
Theorem \ref{thm_main-conjugated_poly-decay}.
\end{proof}

\section{Further comments and variants}
\label{subsec_further_comments_conjugated}
Let us first show that Theorem \ref{thm_main-conjugated_poly-decay} cannot be extended to flat $C^\infty$ phases:

\begin{proposition}
\label{prop_flatness_obstruction_Pisot}
Let $0<\lambda<1/2$, and suppose that $\lambda^{-1}$ is a Pisot
number.
Let $s=r+\alpha>1$, where $r\in\mathbb{N}$ and
$0<\alpha\leq1$, and let $h\in C^{r,\alpha}([0,1])$ be a
diffeomorphism such that
$h'\equiv\eta
\text{ on }
X_\lambda:=X_{\Phi_\lambda}$,
for some $\eta\neq0$. 

If, for some $\tau>0$,
$\left|\widehat{h\mu_\lambda}(\xi)\right|
\lesssim
|\xi|^{-\tau},\,
\xi \in \mathbb{R}$,
then
\[
\tau\leq\frac{\log 2}{(-\log \lambda)\cdot  s}.
\]
In particular, if $h\in C^\infty([0,1])$ and $h'$ is constant on
$X_\lambda$, then  $h\mu_\lambda$ cannot have polynomial Fourier decay.
\end{proposition}

\begin{proof}
Write
$\nu:=h\mu_\lambda.$
Composing $h$ with an affine map does not affect whether $\nu$ has
polynomial Fourier decay. We may therefore assume that
$h(0)=0
 \text{ and }
\eta>0.$

We first claim that
\begin{equation}
\label{eq_flat_Taylor_at_zero}
h(x)=\eta x+O(x^s)
\qquad\text{as }x\downarrow0.
\end{equation}
If $r=1$, this follows directly from the $\alpha$-H\"older
continuity of $h'$ and the identity $h'(0)=\eta$. If $r\geq2$, then,
since $X_\lambda$ is perfect and $h'$ is constant on $K_\lambda$,
differentiating along points of $K_\lambda$ gives
$h''(x)=0
\text{ for every }x\in X_\lambda.$
Repeating this argument inductively yields
$h^{(j)}(x)=0,\,
x\in X_\lambda,\,2\leq j\leq r.$
In particular, this is true for $x=0$. So,
 \eqref{eq_flat_Taylor_at_zero} follows from Taylor's theorem and
the $\alpha$-H\"older continuity of $h^{(r)}$.

For $n\geq1$, let
$X_{\lambda,n}:=\lambda^nX_\lambda$
be the leftmost level-$n$ cylinder. We next construct a smooth
localization to $h(K_{\lambda,n})$. Since
\[
X_{\lambda,n}\subset[0,\lambda^n] \text{ and }
X_\lambda\setminus X_{\lambda,n}
\subset
\left[(1-\lambda)\lambda^{n-1},1\right], \text{ while }
h(x)=\eta x+o(x)
\text{ as }x\downarrow0,
\]
we may choose constants $A$ and $B$ such that
\[
\eta<A<B<\eta\frac{1-\lambda}{\lambda}
\text{ and, for every sufficiently large } n\,,
h(X_{\lambda,n})\subset[0,A\lambda^n] \text{ and }
h(X_\lambda\setminus X_{\lambda,n})
\subset[B\lambda^n,\infty).
\]
Choose $\chi\in C_c^\infty(\mathbb{R})$ such that
\[
\chi\equiv1
\text{ on }[0,A] \text{ and }
\chi\equiv0
\text{ on }[B,\infty), \text{ and define }
\chi_n(x):=\chi(\lambda^{-n}x).
\]
On $\operatorname{supp}(\nu)=h(X_\lambda)$, the function $\chi_n$
therefore agrees with the indicator function of
$h(X_{\lambda,n})$. Since
$\mu_\lambda|_{X_{\lambda,n}}=
2^{-n}(S_0^n\mu_\lambda),$, where
$S_0(x):=\lambda x,$
we obtain
\begin{equation}
\label{eq_localized_left_cylinder}
\chi_n\nu
=
2^{-n}(h\circ S_0^n)\mu_\lambda.
\end{equation}

Suppose now that
\begin{equation}
\label{eq_assumed_power_decay_flat_obstruction}
|\widehat{\nu}(\xi)|
\leq
C(1+|\xi|)^{-\tau}.
\end{equation}
We claim that, for every $L\geq1$,
\begin{equation}
\label{eq_localized_power_decay_flat_obstruction}
\left|\widehat{\chi_n\nu}(\xi)\right|
\leq
C_L\left(
(1+|\xi|)^{-\tau}
+
(1+\lambda^n|\xi|)^{-L}
\right),
\end{equation}
where $C_L$ is independent of $n$.

Indeed,
\[
\widehat{\chi_n\nu}(\xi)
=
\int_{\mathbb{R}}
\widehat{\chi_n}(u)\widehat{\nu}(\xi-u)\,du, \text{ and }
\widehat{\chi_n}(u)
=
\lambda^n\widehat{\chi}(\lambda^nu),
\]
and change of variables  $v=\lambda^nu$,  gives 
$\left|\widehat{\chi_n\nu}(\xi)\right|
\leq
\int_{\mathbb{R}}
|\widehat{\chi}(v)|
\left|
\widehat{\nu}\left(\xi-\lambda^{-n}v\right)
\right|
\,dv.$
On the region
$|v|\leq\frac12\lambda^n|\xi|,$
we have
$\left|\xi-\lambda^{-n}v\right|
\geq
\frac12|\xi|,$
so \eqref{eq_assumed_power_decay_flat_obstruction} bounds the
corresponding integral by
$C'(1+|\xi|)^{-\tau}.$
On the complementary region, we use
$|\widehat{\nu}|\leq1$
and the rapid decay of $\widehat{\chi}$ to obtain
$\int_{|v|>\lambda^n|\xi|/2}
|\widehat{\chi}(v)|\,dv
\leq
C_L(1+\lambda^n|\xi|)^{-L}.$
This proves \eqref{eq_localized_power_decay_flat_obstruction}.

By Erd\H{o}s' Pisot argument, there exists $c_0>0$ such that
\begin{equation}
\label{eq_Pisot_resonant_frequencies_flat_obstruction}
\left|\widehat{\mu_\lambda}(t_N)\right|
\geq
c_0,
\qquad
t_N
:=
\frac{\lambda^{-N}}{1-\lambda},
\qquad
N\geq1.
\end{equation}

Write $d_\lambda:=\frac{\log 2}{-\log \lambda}$ Suppose, towards a contradiction, that
$\tau>\frac{d_\lambda}{s}.$
The function
$a\mapsto
\frac{a\log2}{(1+a)\log(1/\lambda)}$
is continuous, and at
$a=\frac{1}{s-1}$
its value is
$\frac{d_\lambda}{s}.$
We may therefore choose $a>0$ such that
\begin{equation}
\label{eq_choice_a_flat_obstruction}
a(s-1)>1
\end{equation}
and
\begin{equation}
\label{eq_choice_a_decay_flat_obstruction}
a\log2
<
\tau(1+a)\log(1/\lambda).
\end{equation}
Set
$n_N:=\lfloor aN\rfloor$
and
$\xi_N
:=
\frac{t_N}{\eta\lambda^{n_N}}.$
By \eqref{eq_localized_left_cylinder},
\[
2^{n_N}\widehat{\chi_{n_N}\nu}(\xi_N)
=
\int
e^{-2\pi i\xi_Nh(\lambda^{n_N}x)}
\,d\mu_\lambda(x).
\]
Moreover,
$\eta\lambda^{n_N}\xi_N=t_N.$
Using \eqref{eq_flat_Taylor_at_zero} and
$|e^{iu}-e^{iv}|\leq|u-v|,$
we obtain
\begin{align}
&
\left|
2^{n_N}\widehat{\chi_{n_N}\nu}(\xi_N)
-
\widehat{\mu_\lambda}(t_N)
\right|
\nonumber\\
&\qquad\leq
2\pi|\xi_N|
\int
\left|
h(\lambda^{n_N}x)
-
\eta\lambda^{n_N}x
\right|
\,d\mu_\lambda(x)
\nonumber\\
&\qquad\lesssim
|\xi_N|\lambda^{sn_N}
=
t_N\lambda^{(s-1)n_N}
\lesssim
\lambda^{-N+(s-1)n_N}.
\label{eq_resonance_approximation_flat_obstruction}
\end{align}
By \eqref{eq_choice_a_flat_obstruction}, the last expression tends to
zero as $N\to\infty$. Combining this with
\eqref{eq_Pisot_resonant_frequencies_flat_obstruction}, we obtain
\begin{equation}
\label{eq_localized_lower_bound_flat_obstruction}
\left|
\widehat{\chi_{n_N}\nu}(\xi_N)
\right|
\geq
\frac{c_0}{2}\,2^{-n_N}
\end{equation}
for all sufficiently large $N$.

On the other hand, applying
\eqref{eq_localized_power_decay_flat_obstruction} at $\xi=\xi_N$
gives
\[
\left|
\widehat{\chi_{n_N}\nu}(\xi_N)
\right|
\lesssim_L
(1+|\xi_N|)^{-\tau}
+
(1+\lambda^{n_N}|\xi_N|)^{-L}.
\]
Since
$|\xi_N|
\asymp
\lambda^{-(N+n_N)}$
and
$\lambda^{n_N}|\xi_N|
\asymp
\lambda^{-N},$
we obtain
\begin{equation}
\label{eq_localized_upper_bound_flat_obstruction}
\left|
\widehat{\chi_{n_N}\nu}(\xi_N)
\right|
\lesssim_L
\lambda^{\tau(N+n_N)}
+
\lambda^{LN}.
\end{equation}
Choose $L$ sufficiently large that
$L\log(1/\lambda)>a\log2.$
Then
$2^{n_N}\lambda^{LN}\rightarrow0.$
Moreover, by \eqref{eq_choice_a_decay_flat_obstruction},
$2^{n_N}\lambda^{\tau(N+n_N)}
\rightarrow0.$
Multiplying \eqref{eq_localized_upper_bound_flat_obstruction} by
$2^{n_N}$ therefore gives
$2^{n_N}
\left|
\widehat{\chi_{n_N}\nu}(\xi_N)
\right|
\rightarrow0,$
contradicting
\eqref{eq_localized_lower_bound_flat_obstruction}. Hence
\[
\tau\leq\frac{d_\lambda}{s}.
\]
\end{proof}

We conclude with two variants of the  construction proving Theorem \ref{thm_main-conjugated_poly-decay}. The first
concerns positive results for smooth phases. By Proposition \ref{prop_flatness_obstruction_Pisot},  Theorem
\ref{thm_main-conjugated_poly-decay} cannot produce a  \(C^\infty\) conjugacy to $\Phi_\lambda$ with the asserted properties. Note that by replacing
the exponentially decaying perturbations used in section \ref{Section: Canot set construction} by
superexponentially decaying ones, the same geometric construction 
produces a  \(C^\infty\) phase. However, the number of effective random scales in the
\(N\)-th frequency shell then decreases from order \(N\) to order
\(N/\log N\). A similar probabilistic argument then gives
stretched-exponential decay in \(\log|\xi|\). While slower than polynomial, it is still faster
than every negative power of \(\log|\xi|\).

\begin{proposition}
\label{prop_smooth_stretched_log_decay}
Fix \(0<\lambda<1/2\) and \(0<\vartheta<1\). There exist a
\(C^\infty\) \(\lambda\)-pseudo-affine IFS \(\Phi\) and a
\(C^\infty\) diffeomorphism \(h:[0,1]\to[0,1]\), conjugating
\(\Phi_\lambda\) to \(\Phi\) on their attractors, such that
$h'\equiv\eta
\text{ on }X_{\Phi_\lambda}$, 
for some \(\eta>0\). Moreover, the uniform stationary measure
\(\nu=h_*\mu_\lambda\) satisfies, for some \(c,C>0\),
\begin{equation}
\label{eq_smooth_stretched_log_decay}
\bigl|\widehat{\nu}(\xi)\bigr|
\leq
C\exp\left(
-c\bigl(\log(2+|\xi|)\bigr)^\vartheta
\right),
\qquad
\xi\in\mathbb R,
\end{equation}
\end{proposition}

\begin{proof}
Write
$\beta:=\lambda^{-1},$
and choose
$0<\delta<
\frac{1-\vartheta}{16\log\beta}.$
We repeat the construction of Section
\ref{Section: Canot set construction}, replacing the sequence in
\eqref{eq_u_n-(r,alpha)-case} by
\[
a_n
:=
c_0\lambda^{\delta n\log(n+2)},
\qquad
d_n(\omega)
:=
(1-\lambda)\lambda^n
\bigl(1+a_nY_n(\omega)\bigr),
\]
where \(c_0>0\) is sufficiently small. Define
\(L_n,\ell_n,\widetilde d_n\), and \(X_\omega\) exactly as in
\eqref{eq_Ln_conjugated}--\eqref{eq_Xomega_construction}.
For every \(M\geq1\), we have
$a_n\lesssim_M \lambda^{Mn}.$
The estimates in the proof of Proposition
\ref{prop_construction_IFS_conjugated_to_affine_finite_regularity}
therefore give
$|\theta_n|\lesssim_M\lambda^{Mn}$
for every \(M\geq1\). The \(C^\infty\) parts of Theorems
\ref{thm_classification_pAffine_IFS} and
\ref{them_criterion_for_conjgation} now give a \(C^\infty\)
\(\lambda\)-pseudo-affine IFS \(\Phi_\omega\) and a \(C^\infty\)
address-preserving conjugacy \(h_\omega\), with
$h_\omega'\equiv L_0(\omega)^{-1}
\text{ on }\operatorname{supp}(\mu_\lambda).$
The coding and product formulas from the preceding construction also
remain unchanged. In particular, if \(\eta_\omega\) denotes the law of
$\sum_{n=0}^{\infty}\epsilon_nd_n(\omega),$
then
\begin{equation}
\label{eq_smooth_eta_product}
\bigl|\widehat{\eta_\omega}(\xi)\bigr|
=
\prod_{n=0}^{\infty}
\left|
\cos\bigl(\pi\xi d_n(\omega)\bigr)
\right|,
\qquad
\widehat{\nu_\omega}(\xi)
=
\widehat{\eta_\omega}
\left(\frac{\xi}{L_0(\omega)}\right).
\end{equation}

It remains to prove the Fourier-decay estimate. For \(N\geq1\), let
\[
\mathcal A_N
:=
\left[
\frac{\beta^N}{1-\lambda},
\frac{\beta^{N+1}}{1-\lambda}
\right],
\]
and redefine the set of effective random scales by
\begin{equation}
\label{eq_smooth_good_scales}
G_N
:=
\left\{
0\leq m\leq N:
a_{N-m}\beta^m\geq1
\right\}.
\end{equation}
Writing \(n=N-m\), the condition in
\eqref{eq_smooth_good_scales} becomes
$c_0\beta^{N-n-\delta n\log(n+2)}\geq1.$
If
$0\leq n\leq
\frac{N}{4\delta\log(N+2)},$
then, for all sufficiently large \(N\),
$n\leq\frac N4,
\,
\delta n\log(n+2)\leq\frac N4.$
Hence every such \(n\) gives an element \(m=N-n\) of \(G_N\), and
therefore
\begin{equation}
\label{eq_smooth_number_good_scales}
|G_N|
\geq
\frac{N}{8\delta\log(N+2)}
\end{equation}
for all sufficiently large \(N\).

Fix \(\xi\in\mathcal A_N\) and write
$(1-\lambda)\xi=\beta^Ny,
\,
1\leq y\leq\beta.$
For \(m\in G_N\), setting \(n=N-m\), the corresponding factor in
\eqref{eq_smooth_eta_product} has the form
$\left|
\cos\left(
\pi y\beta^m
\bigl(1+a_{N-m}Y_{N-m}\bigr)
\right)
\right|.$
Moreover,
$a_{N-m}y\beta^m\geq1.$
So, Lemma
\ref{lem_one_scale_moment_conjugated} gives, for every \(p\geq1\),
\[
\mathbb E
\bigl|\widehat{\eta_\omega}(\xi)\bigr|^p
\leq
\rho_p^{|G_N|},
\qquad
\rho_p:=\sqrt{\frac{9}{2\pi p}}.
\]

Set
$p_N:=N^{1-\vartheta}.$
Then
$\log\rho_{p_N}
=
-\frac{1-\vartheta}{2}\log N+O(1),$
and \eqref{eq_smooth_number_good_scales} yields
\begin{equation}
\label{eq_smooth_moment_exponent}
\log
\mathbb E
\bigl|\widehat{\eta_\omega}(\xi)\bigr|^{p_N}
\leq
-\frac{1-\vartheta}{16\delta}N+o(N),
\end{equation}
uniformly for \(\xi\in\mathcal A_N\).
Choose
\[
0<\kappa<
\frac{1-\vartheta}{16\delta}-\log\beta.
\]
Let \(\mathcal Z_N\subset\mathcal A_N\) be a mesh with spacing at most
$h_N
:=
\frac{1}{6\pi}e^{-\kappa N^\vartheta}.$
Then
$|\mathcal Z_N|
\lesssim
\beta^Ne^{\kappa N^\vartheta}.$
Let \(F_N\) be the event that
\[
\bigl|\widehat{\eta_\omega}(\zeta)\bigr|
>
\frac12e^{-\kappa N^\vartheta}
\]
for some \(\zeta\in\mathcal Z_N\). By Markov's inequality and
\eqref{eq_smooth_moment_exponent},
\[
\begin{aligned}
\mathbb P(F_N)
&\lesssim
\beta^Ne^{\kappa N^\vartheta}
2^{p_N}
e^{\kappa p_NN^\vartheta}
\rho_{p_N}^{|G_N|} \\
&=
\exp\left(
-\left(
\frac{1-\vartheta}{16\delta}
-\log\beta-\kappa
\right)N
+o(N)
\right).
\end{aligned}
\]
Here we used
$p_NN^\vartheta=N,
\,
p_N=o(N),
\,
N^\vartheta=o(N).$
Thus
\[
\sum_{N=1}^{\infty}\mathbb P(F_N)<\infty.
\]

By  Borel--Cantelli, almost surely \(F_N\) fails for all
sufficiently large \(N\). Since \(\eta_\omega\) is supported on an
interval of length at most \(1+c_0<3/2\),
$\bigl|\widehat{\eta_\omega}'(\xi)\bigr|\leq3\pi.$
The choice of the mesh therefore implies that, almost surely,
\[
\bigl|\widehat{\eta_\omega}(\xi)\bigr|
\leq
e^{-\kappa N^\vartheta},
\qquad
\xi\in\mathcal A_N,
\]
for every sufficiently large \(N\). Since
$N\asymp\log(2+|\xi|)
\text{ on }\mathcal A_N,$
and \(L_0(\omega)\) is bounded above and below uniformly, the scaling
identity in \eqref{eq_smooth_eta_product} gives
\eqref{eq_smooth_stretched_log_decay}. Choosing one \(\omega\) in the
resulting full-probability set completes the proof.
\end{proof}

Theorem \ref{thm_main-conjugated_poly-decay} shows that a suitably
chosen  diffeomorphism may transform a non-Rajchman
self-similar measure into a measure with polynomial Fourier decay,
even though its derivative is constant on the support of the original
measure. This does not mean that non-affinity of
the diffeomorphism alone forces Fourier decay. In fact, the opposite
behaviour is also possible: a smooth non-affine conjugacy, whose
derivative is constant on the relevant Cantor set, may preserve
\emph{non-decay}.

\begin{proposition}
\label{cor:main2}
There exist a self-similar measure
$\nu\in\mathcal P([0,1])$ and a $C^\infty$ diffeomorphism
$h:[0,1]\to[0,1]$ such that:
\begin{enumerate}
    \item both $\nu$ and $h_*\nu$ are not Rajchman;

    \item $h'$ is constant on $\operatorname{supp}(\nu)$, and hence
    $h''|_{\operatorname{supp}(\nu)}=0;$

    \item $h$ is not affine. In fact,
    $\inf_{\ell\ \mathrm{affine}}
    \|h-\ell\|_{C^1([0,1])}>0.$
\end{enumerate}
\end{proposition}

\begin{proof}
Choose an integer $q\geq3$, set
$\lambda:=q^{-1},$
and let $\nu:=\mu_\lambda$ be the uniform self-similar measure
associated to
$\Phi_\lambda
=
\{\lambda x,\lambda x+1-\lambda\}.$
Choose a non-zero finitely supported sequence $(a_n)_{n\geq0}$ such
that
$a_n\in\mathbb Z[q^{-1}]$
and $\sup_n|a_n|$ is sufficiently small. Define
\[
d_n:=(1-\lambda)\lambda^n(1+a_n),
\qquad
L_n:=\sum_{m=n}^{\infty}d_m,
\qquad
\ell_n:=\frac{L_n}{L_0},
\qquad
\widetilde d_n:=\frac{d_n}{L_0}.
\]

Repeating the  construction from Section
\ref{Section: Canot set construction}, these data define a 
homogeneous Cantor set whose level-$n$ cylinders have diameter $\ell_n$,
together with an address-preserving conjugacy $h$ from
$\Phi_\lambda$ to the corresponding pseudo-affine IFS $\Phi$. Since
$a_n=0$ for all sufficiently large $n$, the deviations of the
dynamical proportions from $\lambda$ also vanish at all sufficiently
deep levels. Theorems
\ref{thm_classification_pAffine_IFS} and
\ref{them_criterion_for_conjgation} therefore imply that $\Phi$ and
$h$ may be chosen of class $C^\infty$, and that
$h'\equiv L_0^{-1}
\text{ on }\operatorname{supp}(\nu).$
Moreover, $h_*\nu$ is the law of
\[
\sum_{n=0}^{\infty}\epsilon_n\widetilde d_n, \text{ where the $\epsilon_n$ are independent and uniform on } \{0,1\}.
\]

We next verify directly that both measures are non-Rajchman. For
$\xi_N:=\frac{q^N}{1-\lambda},$
the usual product formula gives
\[
\begin{aligned}
|\widehat{\nu}(\xi_N)|
&=
\prod_{n=0}^{\infty}
\left|
\cos\bigl(\pi q^{N-n}\bigr)
\right| \\
&=
\prod_{m=1}^{\infty}
\left|
\cos\bigl(\pi q^{-m}\bigr)
\right|
=:C_q>0.
\end{aligned}
\]
Indeed,
$1-\left|\cos\bigl(\pi q^{-m}\bigr)\right|
=
O(q^{-2m}),$
so the infinite product defining $C_q$ is positive.

For the image measure, set
$\zeta_N:=\frac{q^NL_0}{1-\lambda}.$
Then
\[
|\widehat{h_*\nu}(\zeta_N)|
=
\prod_{n=0}^{\infty}
\left|
\cos\left(
\pi q^{N-n}(1+a_n)
\right)
\right|.
\]
Since $(a_n)$ is finitely supported and
$a_n\in\mathbb Z[q^{-1}]$, there exists $N_0$ such that, for every
$N\geq N_0$ and every $n\leq N$,
$q^{N-n}(1+a_n)\in\mathbb Z.$
Also, if $N\geq N_0$ and $n>N$, then $a_n=0$. Consequently,
\[
|\widehat{h_*\nu}(\zeta_N)|
=
\prod_{m=1}^{\infty}
\left|
\cos\bigl(\pi q^{-m}\bigr)
\right|
=
C_q>0.
\]
Hence neither $\nu$ nor $h_*\nu$ is Rajchman.

Since $h\in C^\infty$ and $\operatorname{supp}(\nu)$ is perfect, the
identity
$h'\equiv L_0^{-1}
\text{ on }\operatorname{supp}(\nu)$
implies
$h''\equiv0
\text{ on }\operatorname{supp}(\nu).$ Finally, $h$ is not affine. Indeed, the address-preserving conjugacy
fixes $0$ and $1$, so if $h$ were affine, it would be the identity.
This would imply
$\widetilde d_n=(1-\lambda)\lambda^n
\text{ for every }n\geq0.$
Since $a_n=0$ for all sufficiently large $n$, this equality would first
force $L_0=1$, and would then force $a_n=0$ for every $n$, contrary to
our choice of $(a_n)$. Thus $h$ is not affine. Since the space of
affine functions is closed in $C^1([0,1])$, it follows that
$\inf_{\ell\ \mathrm{affine}}
\|h-\ell\|_{C^1([0,1])}>0.$
\end{proof}

\bibliographystyle{amsalpha}
\bibliography{bib}

\providecommand{\bysame}{\leavevmode\hbox to3em{\hrulefill}\thinspace}
\providecommand{\MR}{\relax\ifhmode\unskip\space\fi MR }
\providecommand{\MRhref}[2]{%
  \href{http://www.ams.org/mathscinet-getitem?mr=#1}{#2}
}
\providecommand{\href}[2]{#2}
\begin{thebibliography}{ACWW25}

\bibitem[ABRS26]{ABRHS_pseudo-Affine_IFS}
Amir Algom, Snir {Ben Ovadia}, Federico {Rodriguez Hertz}, and Mario Shannon,
  \emph{How linear can a non-linear hyperbolic {IFS} be?}, Journal d'Analyse
  Math{\'e}matique (2026), To appear.

\bibitem[ACWW25]{Algom2025wu}
Amir Algom, Yuanyang Chang, Meng Wu, and Yu-Liang Wu, \emph{Van der corput and
  metric theorems for geometric progressions for self-similar measures},
  Mathematische Annalen \textbf{393} (2025), 183--214.

\bibitem[AHW22]{algom2021decay}
Amir Algom, Federico~Rodriguez Hertz, and Zhiren Wang, \emph{Logarithmic
  {F}ourier decay for self conformal measures}, J. Lond. Math. Soc. (2)
  \textbf{106} (2022), no.~2, 1628--1661. \MR{4477227}

\bibitem[AHW23]{algom2023polynomial}
\bysame, \emph{Polynomial fourier decay and a cocycle version of {D}olgopyat's
  method for self conformal measures}, arXiv preprint arXiv:2306.01275 (2023).

\bibitem[AK25]{algom2025khalil}
Amir Algom and Osama Khalil, \emph{$l^2$-flattening of self-similar measures on
  non-degenerate curves}, arXiv preprint arXiv:2507.07321 (2025).

\bibitem[AO26]{algom2026orponen}
Amir Algom and Tuomas Orponen, \emph{Uniformly perfect measures on strictly
  convex planar graphs are {$L^2$}-flattening}, Journal de Math{\'e}matiques
  Pures et Appliqu{\'e}es \textbf{213} (2026), 103940.

\bibitem[ARHW21]{algom2020decay}
Amir Algom, Federico Rodriguez~Hertz, and Zhiren Wang, \emph{Pointwise
  normality and {F}ourier decay for self-conformal measures}, Adv. Math.
  \textbf{393} (2021), Paper No. 108096, 72. \MR{4340230}

\bibitem[BB25]{Baker2025banaji}
Simon Baker and Amlan Banaji, \emph{Polynomial {F}ourier decay for fractal
  measures and their pushforwards}, Mathematische Annalen \textbf{392} (2025),
  209--261.

\bibitem[BB26]{baker2026banaji}
\bysame, \emph{Self-similar and self-conformal measures with slow fourier
  decay}, arXiv preprint arXiv:2602.05593 (2026).

\bibitem[BD17]{Bour2017dya}
Jean Bourgain and Semyon Dyatlov, \emph{Fourier dimension and spectral gaps for
  hyperbolic surfaces}, Geom. Funct. Anal. \textbf{27} (2017), no.~4, 744--771.
  \MR{3678500}

\bibitem[BF97]{BedfordFisher1997Ratio}
Tim Bedford and Albert~M. Fisher, \emph{Ratio geometry, rigidity and the
  scenery process for hyperbolic {C}antor sets}, Ergodic Theory Dynam. Systems
  \textbf{17} (1997), no.~3, 531--564. \MR{1452179}

\bibitem[BKS24]{baker2024Kahlil}
Simon Baker, Osama Khalil, and Tuomas Sahlsten, \emph{Fourier decay from $
  l^2$-flattening}, arXiv preprint arXiv:2407.16699 (2024).

\bibitem[Blu99]{Bluhm1999fourier}
Christian Bluhm, \emph{Fourier asymptotics of statistically self-similar
  measures}, Journal of Fourier Analysis and Applications \textbf{5} (1999),
  no.~4, 355--362.

\bibitem[BMPV97]{Bamon1997gogo}
Rodrigo Bam\'on, Carlos~G. Moreira, Sergio Plaza, and Jaime Vera,
  \emph{Differentiable structures of central {C}antor sets}, Ergodic Theory
  Dynam. Systems \textbf{17} (1997), no.~5, 1027--1042. \MR{1477031}

\bibitem[Br{\'e}21]{bremont2019rajchman}
Julien Br{\'e}mont, \emph{Self-similar measures and the {R}ajchman property},
  Ann. H. Lebesgue \textbf{4} (2021), 973--1004. \MR{4315775}

\bibitem[BS23]{Baker2023Sahl}
Simon Baker and Tuomas Sahlsten, \emph{Spectral gaps and {F}ourier dimension
  for self-conformal sets with overlaps}, arXiv preprint arXiv:2306.01389
  (2023).

\bibitem[BY26a]{Banaji2026Qyu}
Amlan Banaji and Han Yu, \emph{Fourier transform of nonlinear images of
  self-similar measures: qualitative aspects}, 2026.

\bibitem[BY26b]{Banaji2025yu}
\bysame, \emph{Fourier transform of nonlinear images of self-similar measures:
  quantitative aspects}, Peking Mathematical Journal (2026), To appear.

\bibitem[Dol98]{Dol1998annals}
Dmitry Dolgopyat, \emph{On decay of correlations in {A}nosov flows}, Ann. of
  Math. (2) \textbf{147} (1998), no.~2, 357--390. \MR{1626749}

\bibitem[Eks16]{Ekstrom2016random}
Fredrik Ekstr{\"o}m, \emph{Fourier dimension of random images}, Arkiv f{\"o}r
  Matematik \textbf{54} (2016), no.~2, 455--471.

\bibitem[Erd39]{Erdos1939ber}
Paul Erd\"{o}s, \emph{On a family of symmetric {B}ernoulli convolutions}, Amer.
  J. Math. \textbf{61} (1939), 974--976. \MR{311}

\bibitem[ES17]{Ekstrom2017jorg}
Fredrik Ekstr\"om and J\"org Schmeling, \emph{A survey on the {F}ourier
  dimension}, Patterns of dynamics, Springer Proc. Math. Stat., vol. 205,
  Springer, Cham, 2017, pp.~67--87. \MR{3775405}

\bibitem[Hoc18]{Hochman2018ICM}
Michael Hochman, \emph{Dimension theory of self-similar sets and measures},
  Proceedings of the {I}nternational {C}ongress of {M}athematicians---{R}io de
  {J}aneiro 2018. {V}ol. {III}. {I}nvited lectures, World Sci. Publ.,
  Hackensack, NJ, 2018, pp.~1949--1972. \MR{3966837}

\bibitem[JS16]{Sahl2016Jor}
Thomas Jordan and Tuomas Sahlsten, \emph{Fourier transforms of {G}ibbs measures
  for the {G}auss map}, Math. Ann. \textbf{364} (2016), no.~3-4, 983--1023.
  \MR{3466857}

\bibitem[Kau84]{Kaufman1984ber}
Robert Kaufman, \emph{On {B}ernoulli convolutions}, Conference in modern
  analysis and probability ({N}ew {H}aven, {C}onn., 1982), Contemp. Math.,
  vol.~26, Amer. Math. Soc., Providence, RI, 1984, pp.~217--222. \MR{737403}

\bibitem[Kha26]{khalil2023exponential}
Osama Khalil, \emph{Exponential mixing via additive combinatorics}, Journal of
  the American Mathematical Society \textbf{39} (2026), no.~3, 781--856.

\bibitem[Li22]{li2018fourier}
Jialun Li, \emph{Fourier decay, renewal theorem and spectral gaps for random
  walks on split semisimple lie groups}, Annales Scientifiques de l''Ecole
  Normale Sup'erieure \textbf{55} (2022), no.~6, 1613--1686.

\bibitem[LPS25]{leclerc2025fourier}
Ga{\'e}tan Leclerc, Sampo Paukkonen, and Tuomas Sahlsten, \emph{Fourier
  dimension in parabolic dynamics}, arXiv preprint arXiv:2505.15468 (2025).

\bibitem[LS22]{li2019trigonometric}
Jialun Li and Tuomas Sahlsten, \emph{Trigonometric series and self-similar
  sets}, J. Eur. Math. Soc. (JEMS) \textbf{24} (2022), no.~1, 341--368.
  \MR{4375453}

\bibitem[LS26]{LequenSahlsten2026}
F{\'e}lix Lequen and Tuomas Sahlsten, \emph{Quantitative fourier decay for
  {Patterson--Sullivan} measures of dimension larger than \(1/2\)}, 2026.

\bibitem[Lyo95]{Lyons1995survey}
Russell Lyons, \emph{Seventy years of {R}ajchman measures}, Proceedings of the
  {C}onference in {H}onor of {J}ean-{P}ierre {K}ahane ({O}rsay, 1993), no.
  Special Issue, 1995, pp.~363--377. \MR{1364897}

\bibitem[Mat15]{Mattila2015Fourier}
Pertti Mattila, \emph{Fourier analysis and {H}ausdorff dimension}, Cambridge
  Studies in Advanced Mathematics, vol. 150, Cambridge University Press,
  Cambridge, 2015. \MR{3617376}

\bibitem[MO23]{Mosquer2023aOlivo}
Carolina~A. Mosquera and Andrea Olivo, \emph{Fourier decay behavior of
  homogeneous self-similar measures on the complex plane}, Journal of Fractal
  Geometry \textbf{10} (2023), no.~1/2, 43--60.

\bibitem[MS18]{Shmerkin2018mos}
Carolina~A. Mosquera and Pablo~S. Shmerkin, \emph{Self-similar measures:
  asymptotic bounds for the dimension and {F}ourier decay of smooth images},
  Ann. Acad. Sci. Fenn. Math. \textbf{43} (2018), no.~2, 823--834. \MR{3839838}

\bibitem[RS20]{Rossi2020Shmerkin}
Eino Rossi and Pablo Shmerkin, \emph{On measures that improve {$L^q$} dimension
  under convolution}, Revista Matem{\'a}tica Iberoamericana \textbf{36} (2020),
  no.~7, 2217--2236.

\bibitem[Sah25]{sahlsten2025fourier}
Tuomas Sahlsten, \emph{Fourier transforms and iterated function systems},
  Recent Developments in Fractals and Related Fields: Conference on Fractals
  and Related Fields IV, {\^I}le de Porquerolles, France, 2022 (Julien Barral,
  Athanasios Batakis, and St{\'e}phane Seuret, eds.), Trends in Mathematics,
  vol. F319, Birkh{\"a}user, Cham, 2025, pp.~297--346.

\bibitem[Sal63]{Salem1963}
Rapha{\"e}l Salem, \emph{Algebraic numbers and fourier analysis}, Heath
  Mathematical Monographs, D. C. Heath and Company, Boston, 1963.

\bibitem[Shm14]{Shmkerin2014Abs}
Pablo Shmerkin, \emph{On the exceptional set for absolute continuity of
  {B}ernoulli convolutions}, Geom. Funct. Anal. \textbf{24} (2014), no.~3,
  946--958. \MR{3213835}

\bibitem[Sol04]{solomyak2004notes}
Boris Solomyak, \emph{Notes on {B}ernoulli convolutions}, Fractal geometry and
  applications. {P}art 1, Proc. Sympos. Pure Math., vol.~72, Amer. Math. Soc.,
  Providence, RI, 2004, pp.~207--230. \MR{2112107}

\bibitem[SS24]{sahlsten2020fourier}
Tuomas Sahlsten and Connor Stevens, \emph{Fourier transform and expanding maps
  on cantor sets}, Amer. J. Math. \textbf{146} (2024), no.~4, 945--982.

\bibitem[Sul88]{Sullivan1987ratio}
Dennis Sullivan, \emph{Differentiable structures on fractal-like sets,
  determined by intrinsic scaling functions on dual {C}antor sets}, The
  Mathematical Heritage of Hermann Weyl, Proceedings of Symposia in Pure
  Mathematics, vol.~48, American Mathematical Society, Providence, RI, 1988,
  Proceedings of the conference held at Duke University, Durham, NC, May
  12--16, 1987, pp.~15--23.

\bibitem[Tsu15]{Tsujii2015self}
Masato Tsujii, \emph{On the {F}ourier transforms of self-similar measures},
  Dyn. Syst. \textbf{30} (2015), no.~4, 468--484. \MR{3430311}

\bibitem[Var23]{varju2023self}
P{\'e}ter~P Varj{\'u}, \emph{Self-similar sets and measures on the line},
  International Congress of Mathematicians, European Mathematical
  Society-EMS-Publishing House GmbH, 2023, pp.~3610--3634.

\end{thebibliography}
\end{document}